\documentclass[12]{amsart}
\usepackage[english]{babel}
\usepackage{epsfig}
\usepackage{graphicx}
\usepackage{epstopdf}
\usepackage{latexsym,amssymb,amsfonts,amsmath, amscd, euscript, MnSymbol}
\usepackage{enumerate}
\usepackage{verbatim}

\def\Power #1 { \powerset(#1) }

\newtheorem{thm}{Theorem}[section]
\newtheorem{cor}[thm]{Corollary}
\newtheorem{lm}[thm]{Lemma}
\newtheorem{con}[thm]{Conjecture}
\newtheorem{prop}[thm]{Proposition}

\newtheorem{exm}[thm]{Example}

\newtheorem{rem}[thm]{Remark}
\newtheorem{pr}[thm]{Problem}
\newtheorem{claim}{Claim}

\newtheorem{comments}[thm]{Comments}
\theoremstyle{definition}
\newtheorem{df}[thm]{Definition}

\newcommand{\R}{\mathbb{R}}
\newcommand{\N}{\mathbb{N}}

\newcommand{\Q}{\mathbb{Q}}

\DeclareMathOperator{\ideal}{Id}
\DeclareMathOperator{\ji}{Ji_{\nabla}}
\DeclareMathOperator{\mi}{Mi}

\def\={ \approx  }

\newcommand{\rto}{\rightarrow}

\begin{document}
\author{Kira Adaricheva}
\address{Department of Mathematical Sciences, Yeshiva University,
245 Lexington ave.,
New York, NY 10016, USA}
\email{adariche@yu.edu}
\author{Maurice Pouzet}
\address{ICJ, Math\'ematiques, Universit\'e Claude-Bernard Lyon1, 43 bd. 11 Novembre 1918, 69622 Villeurbanne Cedex, France and Mathematics \& Statistics Department, University of Calgary, Calgary, Alberta, Canada T2N 1N4}
\email{pouzet@univ-lyon1.fr}

\title{On scattered convex geometries}
\begin{abstract}
A convex geometry is a closure space satisfying the anti-exchange axiom. For several types of algebraic convex geometries we describe when the collection of closed sets is order scattered, in terms of obstructions to the semilattice of compact elements. In particular, a semilattice $\Omega(\eta)$,  that does not appear among minimal obstructions to order-scattered algebraic modular lattices, plays a prominent role in convex geometries case. The connection to topological scatteredness is established in convex geometries of relatively convex sets.
\end{abstract}
\maketitle

\section{Introduction}

We call a pair $(X,\phi)$ of a non-empty set $X$ and a closure operator $\phi: 2^X\rto 2^X$ on $X$ \emph{a convex geometry}\cite{AdaGorTum}, if it is a zero-closed space (i.e. $\overline{\emptyset}=\emptyset$) and $\phi$ satisfies the anti-exchange axiom:
\[
\begin{aligned}
x\in\overline{A\cup\{y\}}\text{ and }x\notin A
\text{ imply that }y\notin\overline{A\cup\{x\}}\\
\text{ for all }x\neq y\text{ in }X\text{ and all closed }A\subseteq X.
\end{aligned}
\]

The study of convex geometries in finite case was inspired by their frequent appearance in modeling various discrete structures, as well as by their juxtaposition to matroids, see \cite{D,EdJa}. More recently, there was a number of publications, see, for example, \cite{A09,SW4,SZ07,S05,W05, adaricheva-nation} brought up by studies in infinite convex geometries. 

A convex geometry is called \emph{algebraic}, if the closure operator $\phi$ is \emph{finitary}. Most of interesting infinite convex geometries are algebraic, such as convex geometries of relatively convex sets, subsemilattices of a semilattice, suborders of a partial order or convex subsets of a partially ordered set. In particular, the closed sets of an algebraic convex geometry form an algebraic lattice, i.e. a complete lattice, whose each element is a join of compact elements. Compact elements are exactly the closures of finite subsets of $X$, and they form a semilattice with respect to the join operation of the lattice.

There is a serious restriction on the structure of an algebraic lattice and its semilattice of compact elements, when the lattice is \emph{order-scattered}, i.e. it does not contain a subset ordered as the  chain of rational numbers $\mathbb Q$. While the description of order-scattered algebraic lattices remains to be an open problem, it was recently obtained in the case of modular lattices. The description is done in the form of \emph{obstructions}, i.e. prohibiting special types of subsemilattices in the semilattice of compact elements.

\begin{thm} \cite{ChaPou0}\label{CP1}
An  algebraic modular lattice is order-scattered iff the semilattice of compact elements is order-scattered and does not contain as a subsemilattice the semilattice $\powerset^{<\omega}(\mathbb N)$ of finite subsets of a countable set.
\end{thm}  

This theorem  was a motivation to the current investigation, due to the fact that convex geometries almost never satisfy the modular law, see \cite{AdaGorTum}. Thus, studying order-scattered convex geometries would open new possibilities for attacking the general hypothesis about order-scattered algebraic lattices. It is known that outside the modular case the list of obstructions must be longer: the semilattice $\Omega(\eta)$ described in \cite{ChaPou2} is order-scattered and isomorphic to the semilattice of compact elements of an algebraic lattice, which is not order-scattered.
As it turns out, $\Omega(\eta)$ appears naturally as a subsemilattice of compact of elements in the convex geometries known as \emph{multichains}. We show in section \ref {omega} that the semilattice of compact elements of a bichain always contains $\Omega(\eta)$, as long as one of chain-orders has the order-type $\omega$ of natural numbers, and an other has the order-type $\eta$ of rational numbers.

\begin{figure}[h]
    \centering
        \includegraphics[width=0.6\textwidth]{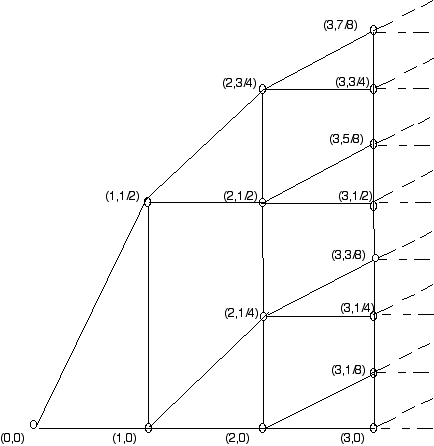}
           \caption{$\Omega(\eta)$}
              \label{O(eta)}
\end{figure}

\vspace{-0.15in}
More generally, in section \ref{section:orderscattered},  we prove in Theorem \ref{fin dim} that any algebraic convex geometry whose semilattice of compact elements $K$ has a finite semilattice dimension will be order-scattered iff $K$ is order-scattered and it does  not have a sub-semilattice isomorphic to  $\Omega(\eta)$.  

As for the other types of convex geometries, we prove the result analogous to modular case. 
It holds true trivially in case of convex geometries of subsemilattices and suborders of a partial order, since order-scattered geometries of these types are always finite, see section \ref{finite}. For the convex geometries of relatively convex sets, we analyze independent sets and reduce the problem to relatively convex sets on a line. As stated in Theorem \ref{CP-like}, the only obstruction in the semilattice of compact elements in this case is $\powerset^{<\omega}(\mathbb N)$. We also discuss the topological issues of the algebraic convex geometries and establish in Theorem \ref{CP-like} that the convex geometry of relatively convex sets is order scattered iff  it is topologically scattered in product topology. This is the result analogous to Mislove's theorem for algebraic distributive lattices \cite{Mis}. Further observations about the possible analogue of Mislove's result in algebraic convex geometries is discussed in section \ref{dist}. In particular, we use some general statements we prove in section \ref{weakly} about weakly atomic convex geometries to build an example of an algebraic distributive lattice that is not a convex geometry.

\section{Preliminaries}

Our terminology agree with \cite{gratzer}. We use the standard notation of $\vee$ and $\wedge$ for the lattice operations of  join and  meet, respectively.
The corresponding notation for infinite join and meet is $\bigvee$ and $\bigwedge$. The lattices where $\bigvee$ and $\bigwedge$ are defined, for arbitrary subsets, are called \emph{complete}. The lattice $L^\sigma$, where operations $\vee$ and $\wedge$ of $L$  are switched, is called \emph{a dual lattice} of $L$. 

Let $L$ be  a lattice. Two elements $x,y$ of  $L$ form \emph{a cover}, denoted by $x\prec y$, if $x <y$ and there is no $z \in L$ such that $x<z<y$. \emph{An interval} in  $L$ is a sublattice of the form $[x,y]:=\{z \in L: x\leq z\leq y\}$, for some $x \leq y$. A lattice, or more generally a poset, is \emph{weakly atomic} if every interval with at least two elements has a cover. 

An element $y \in L$ is called \emph{completely join-irreducible}, if there exists a lower cover $y_*$ of $y$ such that $z < y$ implies $z \leq y_*$, for arbitrary $z \in L$. The set of completely join-irreducible elements is denoted $\ji(L)$. The lattice  $L$ is called \emph{spatial}  if every element is a (possibly, infinite) join  of elements from $\ji(L)$. Dually, one can define the notion of \emph{completely meet-irreducible} elements, and we denote the set of all such elements in  $L$ as $\mi_{\Delta}(L)$. 

Given a non-empty set $X$, a \emph{closure operator} on $X$ is a mapping $\phi: 2^X\rto 2^X$, which is increasing, isotone and idempotent. A subset $Y\subseteq X$ is called \emph{closed} if $Y=\phi(Y)$. The pair $(X, \phi)$ is a \emph{closure system}.  The closure operator $\phi$ is \emph{finitary}, if $\phi(Y)=\bigcup \{\phi(Y'): Y'\subseteq Y, |Y'|<\omega\}$, for every $Y\subseteq X$. The collection of closed sets $Cl(X,\phi)$ forms a complete lattice, with respect to containment. We recall that a lattice $L$ is \emph{algebraic} if it is complete and every element is a (possibly infinite) join of compact elements; these compact elements  form a join-subsemilattice  in $L$.  If $\phi$ is finitary,  the compact elements of the  lattice $L:=Cl(X,\phi)$  are given by $\phi(Y')$, for finite $Y'\subseteq X$, hence $L$ is an algebraic lattice.  The fact that $L$ is algebraic does not ensure that $\phi$ is finitary.  However, every algebraic lattice $L$ is isomorphic to the lattice of closed sets of some finitary closure operator. In fact, if $X \subseteq L$ is the semilattice of compact elements of $L$ then $L$ is isomorphic to  the lattice of closed sets of the closure space $(X,\phi)$, where $\phi(Y)=\{p \in X: p\leq \bigvee Y\}$, for every $Y\subseteq X$. Obviously, this operator $\phi$ is finitary. Alternatively,  $L$ is isomorphic to $\ideal X$, the lattice of ideals of $X$ (recall that an \emph{ideal} of $X$ is a non-empty initial segment which is up-directed). 
Equivalently, $L$ can be thought as a collection of subsets of $X$ closed under arbitrary intersections and  unions of directed families of sets. From topological point of view, $L$ is then a closed subspace of $\mathbf{2}^X$, a topological space with the product topology on the product of $|X|$ copies of two element topological space $\mathbf{2}$ with the discrete topology. Since  $\mathbf{2}^X$ is a compact topological space (due to Tichonoff's theorem), $L$ becomes a compact space, too.

We recall that  if $(X, \phi)$ is a closure system, a subset  $Y\subseteq X$ is  \emph{independent} if $y \not \in \phi(Y\setminus y)$, for every $y \in Y$. A basic property of independent sets is that \emph{a closure  $(X, \phi)$ has no infinite independent subset iff the power set $\powerset(\N)$ ordered by inclusion, is not embeddable into $Cl(X, \phi)$; furthermore, if $\phi$ is finitary  this amounts to the fact that  the semilattice $S$ of compact elements of $Cl(X, \phi)$  does not have $\powerset^{<\omega}(\mathbb{N})$ as a join-subsemilattice}, see for example  \cite{ChaPou, lmp3}.
Finally, let  $X'$ be a subset of $X$. The closure operator $\phi_{X'}$ \emph{induced by} $\phi$ \emph{on}  $X'$ is defined by setting $\phi_{X'}(Y):= \phi(Y)\cap X'$ for every subset $Y$ of $X'$. Clearly $Cl(X', \phi_{X'})= \{Y\cap X': Y\in Cl(X, \phi)\}$. With the definition of induced closure, a subset $Y$ of $X'$ is independent w.r.t. $\phi_{X'}$ iff it is independent w.r.t. $\phi$. Let $\rho_{X'}: Cl(X, \phi)\rightarrow Cl(X', \phi_{X'})$ defined by setting $\rho_{X'}(Y):= Y\cap X'$ and $\theta_{X'}: Cl(X', \phi_{X'})\rightarrow Cl(X, \phi)$ defined by setting $\theta_{X'}(Y):=\phi(Y)$. These two maps are order preserving and the composition map $\rho_{X'}\circ \theta_{X'}$ is the identity  on $Cl(X', \phi_{X'})$. We  will use repeatedly the following result whose proof is left to the reader. 
\begin{lm}\label{lm:embedding}
Let $(X, \phi)$ be  closure and $(X_i)_{i\in I}$ be a family of sets whose union is $X$. Then the map $\rho$ from  $Cl(X, \phi)$ into the direct product  $\Pi_{i\in I}  C(X_i, \phi_{X_i})$ and defined by setting  $\rho (Y):= (Y\cap X_i)_{i\in I}$ is an order-embedding. Furthermore, if $\phi$ is finitary, this map is continuous. \end{lm}
 
 As we  mentioned in the introduction, a \emph{convex geometry} is a pair $(X,\phi)$, where $\phi(\emptyset)=\emptyset$ and $\phi$ satisfies the anti-exchange axiom. A lattice which is isomorphic to the lattice $Cl(X,\phi)$ of a convex geometry will be called a \emph{convexity lattice} (this is a bit different from the convexity lattice of a poset introduced in \cite {birkhoff-bennett}). Apparently, there is no neat charaterization in lattice theoretical terms of convexity lattices, except for finite lattices. With  Theorem \ref{thm:algebraicconvexity} we propose one for   algebraic lattices. 

We recall that the \emph{order-dimension} of a poset $P$, denoted by $dim(P)$,  is  the least cardinal $\lambda$ for which there exist chains $C_i$, $i <\lambda$, such that $P$ is embeddable into the direct product $\Pi_{i<\kappa} C_i$. Alternatively, $dim (P)$ is the least cardinal $\kappa$ such that the order on $P$ is the intersection of $\kappa$ linear orders. There is an important literature about poset dimension, e.g. \cite{trotter} 
We just recall that if $E$ is a set of cardinality $\kappa$ then $dim (\powerset (E))= \kappa$ (H. Komm, see \cite{trotter}) and  if $\kappa$  is infinite, $dim(\powerset^{<\omega}(E))= log_2(log_2(\kappa))$ \cite{kierstead-milner}, where $log_2(\mu)$ is the least cardinal $\nu$ such that $\mu \leq 2^{\nu}$. 
We call a poset $(P,\leq)$ \emph{order-scattered} if it does not have as a sub-poset a chain isomorphic to the chain $\mathbb{Q}$ of rational numbers. We will also refer to the chains isomorphic to  $\mathbb{Q}$ as chains of order-type $\eta$.
Chains isomorphic to the chain of natural numbers $\mathbb{N}$ have order-type $\omega$.
\\
\section{Weakly atomic convex geometries}\label{weakly}

In this section  we prove that weakly atomic convexity lattices are spatial. The result will apply in the next section to produce an algebraic distributive lattice which is not a convex geometry. 

\begin{thm}\label{Weakly}
 A weakly atomic convexity lattice $L$ is spatial. 
In particular, one can choose $Y \subseteq L$, define an anti-exchange operator $\psi$ on $Y$ in such a way that $L$ is isomorphic to $Cl(Y, \psi)$ and  $\psi(y)$ is completely join-irreducible in $Cl(Y,\psi)$, for every $y \in Y$. 
\end{thm}
\begin{proof} We proceed with the following sequence of claims.  The first two hold  in every convexity lattice.
Let  $L:=Cl(X, \varphi)$ where  $(X, \varphi)$ is a convex geometry.

\begin {claim} \label{claim1} If $c \prec d$ in $L$, then $c=X_1$ and $d=X_1 \cup \{x\}$, for some $X_1 \subseteq X$, $x \not \in X_1$.
\end{claim}
Indeed, let $c=X_1=\phi(X_1)\prec d=X_2=\phi(X_2)$. Pick any $x \in X_2\setminus X_1$. Then $X_2=\phi(X_1 \cup \{x\})$. If there is another $y \in X_2\setminus X_1$, $y\not = x$, then $y \in \phi(X_1\cup \{x\})$ implies $x \not \in \phi(X_1\cup \{y\})$. Hence $X_1 < \phi(X_1\cup \{y\}) < \phi(X_1\cup \{x\})=X_2$, a contradiction to $X_1\prec X_2$.

\begin{claim} \label{claim2} Let  $X_1 \prec X_2:=X_1\cup \{x\}$ be  a covering in $L$, then $\phi(\{x\}) \in \ji(L)$.
\end{claim}
Let  $Y:=\phi(\{x\})\cap X_1$. Then $Y=\phi(\{x\})\setminus\{x\}$.  Since it is an intersection of two $\phi$-closed sets,  $Y$ is $\phi$-closed, and $Y\prec \phi(\{x\})$.
If $Z$  is any element of $L$ strictly below $\phi(\{x\})$, then $x \not \in Z$, hence, $Z \leq Y$. This proves that $\phi(x) \in \ji(L)$. 
\begin{claim}\label{claim3} $L$ is spatial.
\end{claim}
Let 
 $Z\in L$.  Set $Z^*:=\bigvee \{X \in\ji(L): X \leq Z\}$. Clearly,  $Z^*\leq Z$. If $Z^* < Z$, then since $L$ is weakly atomic  the interval $[Z^*,Z]$  contains   a cover  $X_1\prec X_2$.  Due to Claim \ref{claim1}, $X_2:=X_1\cup \{x\}$ with $x\not \in X_1$ and, according to Claim \ref{claim2},  $\phi(\{x\}) \in \ji(L)$.  Since  $x\in  Z\setminus  Z^*$ we have $\phi(\{x\})\leq Z$ and $\phi(\{x\})\not\leq Z^*$ contradicting the definition of $Z^*$. Hence, $Z^*=Z$.\\

\noindent Let  $Y:=\{ y\in X: \phi(\{y\}) \in \ji(L)\}$.
Note that $\phi(y)\setminus \{y\}$ is $\phi$-closed, for all $y \in Y$, in particular, $\phi(y_1)\not = \phi(y_2)$, when $y_1\not = y_2$, $y_1,y_2\in Y$.

Let $\psi$ be the closure operator on $Y$ defined for every $Z\subseteq Y$ by setting $\psi(Z):=\phi(Z)\cap Y$. Then $\psi$ satisfies the anti-exchange axiom. Besides, $L$ is isomorphic to $Cl(Y,\psi)$ via the mapping $\rho: A \mapsto A \cap Y$, where $A \in Cl(X,\phi)$. Indeed, $\rho$ is trivially surjective and,  since  $A=\phi(A\cap Y)$ whenever $A=\phi (A)$, a fact which follows from Claim \ref{claim3}, it is one to one. 
\end{proof}

\begin{cor}\label{nice}
In any of the following cases, the convexity lattice $L$ is spatial:
\begin{itemize}
\item[(1)] $L$ is algebraic.

\item[(2)] $L$ is order-scattered.
\end{itemize}
\end{cor}
\begin{proof}
Every  algebraic lattice is weakly atomic   (see, for example \cite{Comp}).  Every   scattered poset is weakly atomic.\end{proof} 

We observe that special type of weakly atomic convex geometries were distinguished in \cite{AdaGorTum} as \emph{the strong convex geometries}: in addition, the latter are atomistic and dually spatial. According to Corollary \ref{nice} (1), all algebraic convex geometries enjoy \emph{the geometric description}, which is equivalent notion of being algebraic and spatial \cite{SW11}.

\begin{thm}\label{thm:algebraicconvexity}An algebraic lattice $L$ is a convexity lattice iff it is spatial  and for every $y\in L$, $u,v \in \ji_{\nabla}(L)$:

\begin{equation}
y<y\vee u=y\vee v \; \text{implies}\;  u=v. 
\end{equation} 
\end{thm}

\begin{proof}
Suppose $L$ is a convexity lattice, i.e. $L\cong Cl (X,\phi)$ for convex geometry $(X,\phi)$. Due to algebraicity, it is weakly atomic, thus, Theorem \ref{Weakly} can be applied to conclude that $L$ is spatial. Every $y\in L$ represents closed set $Y=\phi(Y)\subseteq X$. Elements $u,v \in \ji(L)$ are represented by $\phi(x_u)$ and $\phi(x_v)$, for some $x_u,x_v\in X$. The lattice equality $y\vee u=y\vee v$ is now translated to $x_u\in \phi(Y\cup\{x_v\})$ and $x_v\in \phi(Y\cup\{x_u\})$. According to anti-exchange axiom, we must have $x_u=x_v$, or $u=v$.

Vice versa, suppose $L$ is spatial, for which the equality from the statement of Theorem holds. Denote $X:=\ji(L)$ and define closure operator on $X$ by setting  $\phi (Y) = [0,\bigvee Y]\cap X$, for all $Y\subseteq X$. Then $L\cong Cl(X,\phi)$. Moreover, the anti-exchange axiom holds for $\phi$. Indeed, take any $y\in L$, then by isomorphism it corresponds to $\phi$-closed sets $Y\subseteq X$. Consider $u,v \in X$ such that $u,v\not \leq y$. Then $v\leq  y\vee u$ means that $v \in \phi(Y\cup \{u\})$. Either we assume that $u\not \leq y\vee v$ and then the anti-exchange axiom holds, or we get $y<y\vee u=y\vee v$, which implies $u=v$.
Thus, the anti-exchange axiom holds every time we assume $u\not = v$.
\end{proof}

Another example of weakly atomic convex geometry was presented in \cite{A07}. Since it was given in the form of \emph{antimatroid}, i.e. the structure of open sets of convex geometry, we will provide the corresponding definition of \emph{super solvable} convex geometry here.

\begin{df} A convex geometry $C:=(X,\phi)$ is called \emph{super solvable}, if there exists well-ordering $\leq_X$ on $X$ such that, for all $A,B\in Cl(C)$, if $A \not\subseteq B$, then $A\setminus \{a\}$ is $\phi$-closed, where $a:=\min_{\leq_X}(A\setminus B)$  
\end{df}

We note that the corresponding definition of super solvable antimatroid in \cite{A07} is more restrictive in the sense that $X$ is finite. Super solvable antimatroids with such definition appear as the structure associated with special ordering of elements in Coxeter groups. 
%
%
\begin{cor}
If a convex geometry $C:=(X,\phi)$ is super solvable then $Cl(C)$ is spatial.
\end{cor}
Indeed, the property of super solvable convex geometry guarantees that every interval $[D,A]$ in the lattice of closed sets is strongly co-atomic, i.e., for every $B \in [D,A], B\not = A$, there exists $A'\prec A$ such that $B\leq A'$.
In particular, $C$ is weakly atomic, thus, it is spatial.

An example of a finite super solvable convex geometry is given also by the lattice of $\wedge$-subsemilattices $Sub_\wedge(P)$ of a finite (semi)lattice $P$. This follows from result in \cite{HR93}, where it was established for more general (and dual) lattices of closure operators on finite partially ordered sets. We will deal with infinite lattices $Sub_\wedge(P)$ in section \ref{finite}. 

\section{Distributive lattices and convex geometries}\label{dist}

We call a topological space $Y$ \emph{scattered}, if every non-empty subset $S$ of $Y$ has an isolated point, i.e. there exists $y\in S$ and an open set $U$ of $Y$ such that $\{y\}=S\cap U$. 

The following connection was established between topological and order characteristic of lattices in the distributive case.

\begin{thm}\cite{Mis}\label{Mis}
A distributive algebraic lattice is topologically scattered iff it is order-scattered.
\end{thm} 

In this section we collect several observations concerning the following

\begin{pr}\label{problem}
Is it true that every algebraic convex geometry that is order-scattered will also be topologically scattered?
\end{pr}

We want to emphasize that the term of ``topologically scattered"  algebraic lattice $L$, in Mislove's result and Problem \ref{problem}, assumes the product topology induced on $L$ from $2^X$, where $X$ is the set of compacts of $L$.

Our first observation is that any solution to Problem \ref{problem} can not be a generalization of Mislove's result. For this, we just need to give an example of an algebraic distributive lattice that cannot be the lattice of closed sets of any convex geometry.


\begin{exm}\label{not conv geo}
\end{exm}
Consider the set $L^*$ of cofinite subsets of  a countable set $X$, with the empty set added. Ordered by inclusion,  $L^*$ is a complete lattice. Clearly it is distributive and algebraic, in fact, every element of $L^*$ is compact. Besides, it is order-scattered, thus, weakly atomic.
On the other hand, there is no completely join-irreducible elements, hence, $L^*$ cannot be a convex geometry, due to Corollary \ref{nice}(2).\\

Our next observation is about multiple possibilities to define the product topology on the same lattice. 
Every time there is an embedding  of a complete lattice $L$ into $\mathbf{2}^X$, for some set $X$, we may think of $L$ as a topological space, whose topology inherits the product topology of $\mathbf{2}^X$. Presumably, there are different ways of such representations of $L$. 

When we have  a convex geometry $C:=(X,\phi)$ on set $X$, we have a natural embedding of $C$ into $\mathbf{2}^X$. Thus, saying about topological scatteredness of a convex geometry, we may assume the topology inherited from $\mathbf{2}^X$. 
Not every order-scattered convex geometry is topologically scattered with respect to this embedding, even when its convexity lattice is distributive.

\begin{exm}\label{not top scat}
\end{exm}
Let $P_{<\infty}$ be the convex geometry defined on a countable set $X$, whose closed sets are all finite subsets of $X$ and $X$ itself. Evidently, the lattice of closed sets is distributive, but not algebraic (this convex geometry does not have compact elements at all). On the other hand, it is dually algebraic (in fact the dual lattice is isomorphic to the lattice $L^*$ of Example \ref {not conv geo}). It is order-scattered, but  not topologically scattered. Indeed,  let  $S:=P_{<\infty}\setminus\{X\}$. Then, 
every non-empty open set  $U$ of $\powerset(X)$ has more than one point of intersection with $S$. Indeed, 
 there are finite $X_1,X_2 \subseteq X$ such that $U$ contains $\{Y\subseteq X : X_1 \subseteq Y,X_2 \subseteq X \setminus Y\}$. Then,  every finite $Z\subseteq X$ that has $X_1$ as a subset, and avoids $X_2$ is in $U\cap S$. 


In the next example we demonstrate that the same lattice $L$ may be topologically scattered or not topologically scattered depending on the choice of $X$ and embedding of $L$ into $2^X$. Before, we will need the following observation.

\begin{prop}\label{dual}
If lattice $L$ is topologically scattered with respect to $2^X$, then its dual lattice $L^\sigma$ is topologically scattered with respect to the same $2^X$.
\end{prop}
\begin{proof}
Observe that the complement operation in $2^X$ is continuous with respect to the product topology.
\end{proof}

\begin{exm}\label{diff reps}
\end{exm}
Let $L^*$ be the  lattice of co-finite subsets of  a countable set $X$, with the empty set added to make it complete (see Example \ref{not conv geo}). Then it is dual to $P_{<\infty}$ from Example \ref{not top scat}. Hence, $L^*$ is not topologically scattered with respect to $2^X$, due to Proposition \ref{dual} and Example \ref{not top scat}. On the other hand, $L^*$ is an algebraic distributive lattice, and all its elements are compacts. Associating every element $x \in L^*$ with the ideal $[0,x] \subseteq L^*$, we have an embedding of $L^*$ into $2^{L^*}$, and with respect to this embedding, $L^*$ is topologically scattered, due to Mislove's theorem \ref{Mis}. \\

Finally, despite the fact that not every distributive algebraic lattice is a convex geometry (Example \ref{not conv geo}), one can think of it as an infinite version of \emph{antimatroid}. We recall that an antimatroid stands for the collection of open sets of a convex geometry, thus, it forms a lattice dual to the lattice of closed sets.

\begin{prop}\label{alg-dist}
If $L$ is an algebraic distributive lattice, then it is isomorphic to an antimatroid. 
\end{prop}
\begin{proof}
It is equivalent to show that $L$ is dually isomorphic to the lattice of closed sets of a convex geometry.
As every algebraic lattice,  $L$ has enough completely meet-irreducible elements, that is every element of $L$ is a meet of elements from $M:= \mi_{\Delta} (L)$,  the set of completely meet-irreducible elements of $L$ (see for example \cite{Comp}). Define a closure on $M$ as follows: for any $Y \subseteq M$, set $\phi(Y):=\{m \in M:  m \geq \bigwedge Y\}$. Then $\phi$ is a closure operator on $M$, and the lattice $Cl(M, \phi)$ of closed sets of $\phi$ is isomorphic to $L^\sigma$.
We claim that $\phi$ is anti-exchange. Indeed,  let $A\in Cl(M, \phi)$, $x \in M\setminus A$ and  $x \not = y$ such that $x \in \phi(A \cup \{y\})$. Then $x \geq a\wedge y$, that is  $x=x\vee (a \wedge y)$, where $a =\bigwedge A$. Due to distributivity, $x=(x\vee a)\wedge (x\vee y)$. Thus, $x=x\vee a$ or $x=x\vee y$, since $x$ is meet-irreducible. But since $x\not \in A$,  $x \not \geq a$, hence, $x> y$. Since $x\not \in A$, $y\not \in A$. If   $y \in \phi(A \cup \{x\})$ then by the same token we obtain that $y>x$ which is impossible. Hence,  $\phi$ satisfies the anti-exchange axiom.
\end{proof}

We now turn back to Problem \ref{problem}.
The following two sections provide the partial positive confirmations, when considering some special types of convex geometries.

\section{Relatively convex sets}
Let $V$ be a real vector space and $X\subseteq V$. Let  $Co(V,X)$ be the collection of sets $C\cap X$, where $C$ is a convex subset of $V$. Ordered by inclusion, $Co(V,X)$ is an algebraic convex geometry. Several publications are devoted to this convex geometry \cite{A04,A08,beagley, Be}.

%
%
%
%

The main goal of this section is to prove the following result. 
\begin{thm}\label{CP-like} The following properties are equivalent for a convex geometry $L:=Co(V,X)$.
\begin{enumerate}[{(i)}]

\item $L$ is topologically scattered;
\item $L$ is order-scattered;
\item The semilattice $S$ of compact elements of $L$ is order-scattered and does not have a join subsemilattice isomorphic to $\powerset ^{<\omega}(\mathbb N)$; 
\item $X$ is included into a finite union of lines and on each line $\ell$ with an orientation, the order on the points of $X$ is scattered. 

\end{enumerate}
\end{thm}
The equivalence between $(i)$ and $(ii)$ is an analogue of Mislove's result \cite{Mis}, whereas the equivalence between $(ii)$ and $(iii)$ is an analogue of  Theorem \ref{CP1}.

First, we start from the analysis of independent subsets of $Co(V, X)$.  Set $Co(V):=Co(V, V)$, denotes by $conv(Y)$ the  closure of a subset $Y\subseteq V$ in  $Co(V)$ and   call it  the  \emph{convex hull} of $Y$. Clearly the closure induced on $X$ is the closure in $Co(V, X)$. Hence, the independent sets w.r.t. this closure are the independent sets w.r.t. the closure $conv$, that we call \emph{convexly independent sets}, which are included into $X$.

\begin{thm} \label{ind}
The following properties are equivalent for a subset $X$ of $V$.
\begin{enumerate}[{(i)}]
\item $X$ is contained in a finite union of lines;
\item $X$ contains no infinite convexly independent subset;
\item $dim(Co(V,X))$ is finite.
\end{enumerate}
\end{thm}

The proof is elementary.  The proof of the equivalence between $(i)$ and $(ii)$  relies on classical arguments used in the proofs of the famous Erd\"os-Szekeres theorem (see \cite {Mor}). A connection between Erd\"os-Szekeres conjecture and relatively convex sets is Morris \cite{Mor2}, 
see also \cite{beagley}. 

\begin{proof}

\noindent $\neg (i)\Rightarrow \neg (ii)$.
 We suppose first that $V=\mathbb R^2$.

 If $X$ is not contained in the finite union of lines, then one can find  a countable subset $X_1 \subseteq X$ such that no three points
from $X_1$ are on a line. Indeed, pick two points $x_1,x_2$ from $X$ randomly, and if $x_1,\dots,x_k$ are already picked, choose $x_{k+1} \in X$ so that it does not belong to any line that goes through any two points from $x_1, \dots, x_k$.

Now form $F$, the set of $4$-element subsets of $X_1$, and colour elements of $F$ red, if one point of four is in the convex hull  of the others, and colour it blue otherwise. According to the infinite form of Ramsey's theorem, there exists an infinite subset $X_2\subseteq X_1$ such that all four-element subsets of $X_2$ are coloured in one colour. But it cannot be red colour, because, even for a $5$-element subset of points from $X_1$, at least one $4$-element subset would be coloured blue, see \cite{Mor}. Hence, $X_2$ has all $4$-element subsets coloured blue. It follows that  $X_2$ is in infinite independent subset of $X$. Indeed, if any point $x \in X_2$ was in the closure of some  finite subset $X'\subseteq X_2\setminus \{x\}$, then, due to Carath\'eodory property of the plane, $x$ would be in the closure of $3$ points from $X'$, which contradicts the choice of $X_2$.

Now, we show how to reduce the general case to the case above. 
For this purpose, let $Af(V, X)$ be the set  $A\cap X$, where $A$ is an affine subset of $V$. Ordered by inclusion, $Af(V, X)$ is an algebraic geometric lattice, that is an algebraic lattice and, as  a closure system,  it satisfies the exchange property. Every subset $Y$ of $X$ contains an affinely independent subset $Y'$ with the same affine span $S$ as $Y$;  moreover, the size of $Y'$ is equal to $dim_{af}(S)+1$ where $dim_{af}S$, the affine dimension of $S$, is the ordinary dimension of the  translate of $S$ containing $\{0\}$. 

Suppose that $X$ is not contained in a finite union of lines. 
Let $\lambda$ be the least cardinal such that $X$ contains a subset $X'$ such that $X'$ is not contained in a finite union of lines and the  affine dimension of its affine span is $\lambda$. Necessarily, $\lambda\geq 2$. If $\lambda$ is infinite then  $X$ contains an infinite convexely independent subset. Indeed, $X'$ contains an affinely independent subset of size $\lambda +1$ and every affinely independent set is convexely independent. Suppose that $\lambda$ is finite. We proceed by induction on $\lambda$. We may assume with no loss of generality that $X'\subseteq \mathbb R^{\lambda}$. If $\lambda =2$, the first case applies. 

Suppose $\lambda >2$. Let $X''$ be a projection of  $X'$ on an hyperplane $V'$. If $X''$ is not contained in a finite union of lines, then induction yields an  infinite convexely independent subset of $X''$. For each element $a'$ in this subset, select some element  $a$ in $X'$ whose projection is $a'$. The resulting set  is convexely independent. If $X''$ is contained in a finite union of lines, then there is some line such that  its inverse  image in $X'$ cannot be covered by finitely many lines. This inverse image  being a plane, the first case applies. 

$(i)\Rightarrow (iii)$  Let $(X_i)_{i\in I}$ be a family of subsets of $V$ whose union is $X$.   According to Lemma \ref{lm:embedding},  $Co(V, X)$ is embeddable into the direct product  $\Pi_{i\in I} Co(V,X_i)$, thus from the definition of dimension, $$dim (Co(V,X)) \leq dim (\Pi_{i\in I} Co(V,X_i)).$$  As it is well  known, the order-dimension of a product is at most the sum of order-dimensions of its components (see \cite{trotter}). Now, according to Lemma \ref{lm:interval} stated below,  $dim(Co(V, X_i))\leq 2$ if $X_i$ is contained in a line. Thus $dim(Co(V,X))\leq 2\times \vert I\vert$ whenever $X$ is covered by $\vert I \vert$ lines.

$(iii)\Rightarrow (ii)$ If  $A$ is a subset of $X$, $Co(V,A)$ is embeddable into $Co(V,X)$, hence $dim(Co(V, A))\leq dim (Co(V,X)$. If $A$ is  convexly independent  then  $Co(V, A)$ is order isomorphic to $\powerset (A)$ ordered by inclusion, hence, as mentioned in the preliminaries,  $dim(\powerset(A))=\vert A\vert$. Hence, $\vert A\vert \leq dim(Co(V, X)$.

 \end{proof}
 
 Let $ind(X)$ be the supremum of the cardinalities of the convexly independent subsets of $X$ and $line(X)$ be the least number of lines needed to cover $X$. 
The proofs of implications  $(i)\Rightarrow (iii)$ and $(iii)\Rightarrow (ii)$ show that the following inequalities hold. 
\begin{equation}\label{dimension}ind(Co(V,X))\leq dim (Co(V,X))\leq 2\cdot line(X). 
\end{equation}

For more on these parameters, see the paper of Beagley \cite{beagley}. 

Implication $(ii)\Rightarrow (iii)$ in Theorem \ref{ind} has the following corollary pointing out a property which is not shared by many convex geometries.

\begin{cor}
If $X$ contains  convexly independent sets of arbitrary large finite size, then it contains an infinite convexly independent set.
\end{cor}

Let $Co^{<\omega}(\mathbb N)$ be the (semi)lattice of finite intervals  of the chain of natural numbers $\mathbb N$, ordered by inclusion, and let $\powerset^{<\omega}(\mathbb N)$ be the (semi)lattice of finite subsets of $\mathbb N$.
\begin{cor}
If $X$ is infinite, then the semilattice of compact elements of  $Co(V ,X)$ contains either $Co^{<\omega}(\mathbb N)$ or $\powerset^{<\omega}(\mathbb N)$ as a join semilattice.
\end{cor}
\begin{proof}
Apply Theorem  \ref{ind}.  
If $X$ contains an infinite independent subset, then the semilattice of compact elements of $L:=Co(V,X)$ will have a semilattice isomorphic to $ \powerset^{<\omega}(\mathbb N)$. Otherwise, $X$ must be covered by finitely many lines. If $X$ is infinite, then one of the lines will have infinitely many points from $X$. Choose an origin and an orientation on that line. Then one can find either an increasing or a decreasing infinite  countable sequence of elements of $X$ on that line. Hence, the semilattice of compact elements of $L$ has $Co^{<\omega}(\mathbb N)$ as a subsemilattice.
\end{proof}

In the next statement, a set  $X$ of points on a  line $L$ in $V$ can be thought as a subset of the line of real numbers.
\begin{prop}\label{line}
Let $X$ be a set of points on a line $L$ in $V$.  The following are equivalent:
\begin{enumerate}[{(i)}]
\item $Co(V,X)$ is topologically scattered;
\item $Co(V,X)$ is order-scattered;
\item the semilattice $S$ of compact elements of $Co(V,X)$ is order-scattered;
\item $X$ is order-scattered in $L$.
\end{enumerate}
\end{prop}

 The equivalence holds in the more general case of a chain $C:=(X, \leq)$, and the lattice $Int(C)$ of intervals of $X$ standing for $Co(V,X)$.

 We recall that if $C:=(X, \leq )$ is a chain,  a subset $A$ of $X$ is an \emph{interval} if $x, y \in A$, $z\in X$ and $x\leq z\leq y$ imply $z\in A$. The set  $Int(C)$ of intervals of $C$, ordered by inclusion,  forms an algebraic closure system satisfying  the anti-exchange axiom. The join-semilattice of compact elements is made of the closed intervals $[a, b]:=\{z\in X: a\leq z\leq b\}$ where $a, b\in X$, $a\leq b$. 
Let $I(C)$, resp. $F(C)$, be the set of initial  resp. final, segments of $C$.  Ordered by inclusion these sets are also algebrac closure systems satisfying the anti-exchange axioms. Let $A\subseteq X$; set $\downarrow A:=\{y\in X: y\leq a\;  \text {for some}\; a\in A\}$. Define similarly $\uparrow A$; these sets are the closure of $A$ w.r.t. $I(C)$ and $F(C)$. As subsets of $\powerset (X)$,  the sets $I(C)$, $F(C)$ and  $Int(C)$ inherit of the product topology. 

The proof of Proposition \ref{line} is based on the following lemma, on some basic properties of scattered topological spaces and on a similar property for the collection of initial segment of a chain given in Proposition \ref{lem:scattered chain} below .

\begin{lm}\label {lm:interval}The map $f$ from  from $Int(C)$ into the direct product $I(C)\times F(C)$ defined by $f(A):=(\downarrow A, \uparrow A)$ is an embedding. The map $g$ from $I(C)\times F(C)$ into $Int(C)$ defined by $g(I,J):=I\cap J$ is surjective,  order-preserving and continuous. \end{lm}

The proof is straightforward. Note that if  $X$ is infinite,  the map $f$ is not continuous.

We recall some  basic results on scattered spaces. 
\begin {lm}\label{product}
\begin{enumerate}
\item If  $Y$ is  a subset of a scattered topological space $X$ then  $Y$ 
is scattered w.r.t. the induced topology; 
\item  If $Y$ is a continuous image of a compact scattered topological space $X$ then $Y$ is scattered;
\item If $X$ is the union of finitely many scattered subspaces then $X$ is scattered; 
\item If $Y_i$, $i=1,...,n$, are topologically scattered spaces, then $Y:=\Pi_{i<n} Y_i$, with the product topology on $Y$, is topologically scattered  too.
\end{enumerate}
\end {lm}
The  proof of $(2)$, quite significant,  is due to A.Pelczynski and Z.Semadeni \cite {Pel};  the proofs of the other items are immediate. 

The following result goes back to Cantor and Hausdorff. 

\begin{prop} \label{lem:scattered chain}The following properties for a chain $C$ are equivalent:
\begin{enumerate}
\item $I(C)$ is topologically scattered;
\item $I(C)$ is order-scattered;
\item $C$ is order-scattered. 
\end{enumerate}
Furthermore, if  a complete chain $D$ is order-scattered, then it is isomorphic to $I(C)$ where $C$ is some scattered chain. 
\end{prop}

As  a corollary we get the following well-known statement.
\begin{cor}\label{top->order}
Every algebraic lattice $L$ that is topologically scattered (in the product topology of $\mathbf{2}^X$)  is also order-scattered.
\end{cor}
\begin{proof}[Proof of Proposition \ref{line}]
The implications $(i)\Rightarrow (ii)\Rightarrow (iii)$ are valid in any algebraic closure system. 
Implication $(i)\Rightarrow (ii)$ follows from Corollary \ref{top->order}. Implication $(ii)\Rightarrow (iii)$ is trivial. 
Implication $(iii)\Rightarrow (iv)$: the set $X$ being  thought as a subset of the line of real numbers, set $C:=(C, \leq)$ with the order induced by the natural order on the reals; suppose that   $X$ contains a  subset $A$ of order type $\eta$. Pick  $a\in A$; the  intervals $[a, b]$ of $X$ with $a\leq b\in A$ form a chain of  compact  elements of order type $\eta$. Implication $(iv)\Rightarrow (i)$: 
suppose that $C$ is order-scattered. Then according to Proposition \ref{lem:scattered chain}, $I(C)$ and $F(C)$ are topologically scattered. Hence, from  (4) of Lemma \ref{product}, the direct product $I(C)\times F(C)$ is topologically scattered. Since according to Lemma \ref{lm:interval}, $Int(C)$ is the continuous image of  $I(C)\times F(C)$, it is topologically scattered from  (2) of Lemma \ref{product}. 
\end{proof}

\begin{lm}\label{q in product}
The direct product $\Pi_{k<n} C_k$ of finitely many order-scattered posets is order-scattered.
\end{lm}
\begin{proof} Induction by $n$. It is trivial for $n=1$. Suppose it is true for $n'$.  But that there is an embedding of $\mathbb {Q}$ into a  product of $n'+1$ posets $P_k$, $k\leq n'+1$. If for each pair $r<q$ of rationals, we have $r[1]<q[1]$, then $\mathbb{Q}$ can be embedded into $C_1$. If for some $p<q$ we have $p[1]=q[1]$, then interval $[p,q]\simeq \mathbb{Q}$ must be embedded into $C_2\times \dots \times C_{k+1}$. Then, according to hypothesis, it should be embedded into one of $C_2,\dots, C_{k+1}$.
\end{proof}
Let $\ideal P$ be the collection of ideals of a poset. 
\begin{lm}\label{lm:pdscattered}
$\ideal  (P\times Q)$ is isomorphic to $\ideal  P \times \ideal  Q$. En particular, if $(C_k)_{k<n}$ is a finite family of chains then  $\ideal  (\Pi_{k<n} C_k)$ is isomorphic to $\Pi_{k<n} \ideal  (C_k)$.  
\end{lm}
This is well-known, see \cite {PoSiZa2} for an example. 
For the proof of Theorem \ref{CP-like} we  use the following corollary of Lemma \ref{product} and  Lemma \ref{lm:embedding}.
\begin{cor}\label{prod}
Suppose $X \subseteq \bigcup_{i \leq n} X_i \subseteq V$. If every $Co(V,X_i)$ is topologically scattered, then $Co(V,X)$ is topologically scattered.
\end{cor}
\begin{proof}
%
Since the image of $Co(V,X)$ under $\rho$ is a subspace of $\Pi_{i<n} Co(V,X_i)$  and,  The map $\rho$ from $Co(V,X)$ into $\Pi_{i<n} Co(V,X_i)$ defined in  Lemma \ref{lm:embedding} is continuous. Due to $(4)$ of Lemma \ref{product}, $\Pi_{i<n} Co(V,X_i)$ is topologically scattered, hence the image of $Co(V,X)$  by $\rho$ is scattered. Since $\rho$ is one-to-one,  $Co(V,X)$ must be topologically scattered as well.
\end{proof}

\begin{proof}[Proof of Theorem \ref{CP-like}]
We note first that implications $(i)\Rightarrow (ii)$ and $(ii)\Rightarrow (iii)$ holds for arbitrary algebraic lattices. 

\noindent$(i)\Rightarrow (ii)$ Corollary \ref{top->order}. 

\noindent $(ii)\Rightarrow (iii)$. If $S$ contains  a join semilattice isomorphic to $\powerset^{<\omega}(\mathbb N)$ then, since $L$ is algebraic, $L$ contains  a join semilattice isomorphic to $\powerset(\mathbb N)$; since $\powerset (\mathbb N)$ contains a copy of the chain of real numbers, $L$ is not order-scattered. 

\noindent $(iii)\Rightarrow(iv)$. Suppose that $(iii)$ holds. Since $L$ is algebraic and $S$  does not contain a join subsemilattice isomorphic to  $\powerset ^{<\omega}(\mathbb N)$  then, as mentioned in the preliminaries,  $X$ cannot contain an infinite independent subset.  Hence, according to Theorem \ref{ind}, $X$ should be covered by  finitely many lines $\ell_i$, $i \leq n$. For $i\leq n$, set $X_i:=X\cap \ell_i$.  Since $Co(V, X_i)$ is a join subsemilattice  of  $Co(V,X)$, it is order-scattered, in particular the order induced by any orientation of $\ell_i$ is scattered. 

 $(iv)\Rightarrow(i)$. Suppose that $(iv)$ holds. Let  $\ell_i$, $i \leq n$ be finitely many lines whose union covers $X$ and such that the order induced by each  orientation on $\ell_i\cap X$ is scattered. For $i\leq n$, set $X_i:=X\cap \ell_i$.  Due to Proposition \ref{line}, $Cov(V, X_i)$ is topologically scattered. Corollary 
\ref{prod} implies that $Co(V,X)$ is  topologically scattered as well. 
\end{proof}

\section{The lattice of subsemilattices and the lattice of suborders}\label{finite}

The convex geometries made of the  subsemilattices of a semilattice and of the suborders of a partially ordered set play an  important role in the studies of convex geometries in general due to their close connection to lattices of quasi-equational theories, see \cite{Ada, A09, AdaGorTum, S05, W05}.

\begin{thm}
If $S$ is an infinite meet-semilattice, then the lattice $Sub_\wedge(S)$ of meet-subsemilattices of $S$ always has a copy of $\mathbb{Q}$. Thus, $Sub_\wedge(S)$ is order-scattered iff $S$ is finite.
\end{thm}
\begin{proof}
As it is well known,  every infinite poset contains either an infinite chain or an infinite antichain. Let $X$ be such a subset of $S$. As it is easy to check, such a  subset is independent w.r.t.   the closure associated with $Sub_\wedge(S)$,  hence  $\powerset({X})$ is embeddable into $Sub_\wedge(S)$.  Since $X$ is infinite, $\powerset({X})$ contains a copy of the real line, hence of  $\mathbb{Q}$. 
%
%
\end{proof}

Similar result holds for the lattice of suborders. For a partially ordered set $\langle P,\leq \rangle$, denote by $S(P)$ the  strict order associated to  $P$, i.e. $S(P)=\{(p,q): p\leq q \text{ and } p\not = q, p,q \in P\}$. Then, $O(P)$, 
the lattice of suborders of $P$,  is the set of  transitively closed subsets of $S(P)$.

\begin{thm}
The lattice of suborders $O(P)$ of a partially ordered set  $\langle P,\leq \rangle$ is order-scattered iff $S(P)$ is finite.
\end{thm}
\begin{proof}
Suppose $S(P)$ is infinite. 
\begin{claim}\label{claim:ramsey}
$O(P)$ contains a suborder of $P$ which is either
\begin{enumerate}
\item an infinite chain, 
or
\item an infinite antichain with an element below or  above all the elements of the antichain, 
or
\item the direct sum of infinitely many $2$-element chains. 
\end{enumerate}
\end{claim}
\noindent{\bf Proof of Claim \ref{claim:ramsey}. }
Let $(x_{0,n},x_{1,n})_{n\in \N}$ be an infinite sequence or elements of $S(P)$. Let $[\N]^2$ be the set of pairs of integers, identified with pairs $(n,m)$, with $n<m$.  Say that two  pairs  $(n, m)$ and $(n',m')$ are equivalent if there is a  map $h$ from $H_{n,m}:=\{x_{i, n}, x_{i,m}: i<2\}$ into $H_{n',m'}:=\{x_{i, n'}, x_{i,m'}: i<2\}$ satisfying $h(x_{i,n}):=x_{i, n'}$ and $h(x_{i,m}):=x_{i, m'}$ for $i<2$ which is an order-isomorphism of $H_{n,m}$ on $H_{n',m'}$,  once these sets are ordered according to $P$. This define an equivalence relation on $[\N]^2$. The number of classes being finite, the infinite version of Ramsey's theorem ensures that we can find an infinite  subset $I$ of $\N$ such that all pairs in $[I]^2 $ are equivalent. If there is a pair $(n, m)\in [I]^2$ with $x_{i,n}$ and $x_{i,m}$ comparable and distinct (for some $i<2$), then this is the same for all the other pairs. In this case we obtain an infinite chain. Exclude this case. If there is a pair $(n, m)\in [I]^2$ with $x_{i,n}= x_{i,m}$, then $x_{i+1,n}$ and $x_{i+1,m}$ must be incomparable and we get an infinite antichain with an element below or  above all the elements of the antichain. Excluding  this case,  there is a pair $(n, m)\in [I]^2$ with $x_{i,n}$ incomparable  to $x_{i,m}$,for $i<2$. If $x_{i,n}$ is comparable  to $x_{i+1,m}$ for some $i<2$, then we get an element below or above all elements of an antichain. This case being excluded, we get a direct sum of infinitely many $2$-element chains, as claimed. \hfill   $\Box$

This claim allows to define an embedding from $\powerset(\mathbb {N})$  into $O(P)$. It follows that $\Q$ is embeddable into $O(P)$. 
\end{proof}

\section{Representation of join-semilattices}

Let $L$ and $L'$ be two join-semilattices. A map $f:L\rightarrow L$ is \emph{join-preserving} if 

\begin{equation} 
f(a\vee b)=f(a)\vee f(b)
\end{equation}

holds for every $a,b\in L$. 

If $L$ and $L'$ are complete lattive, $f$ \emph{preserves arbitrary joins} if 

\begin{equation} 
f(\bigvee X)=\bigvee f(X)
\end{equation}

holds for every subset $X$ of $L$. 

The definitions of \emph{meet-preserving maps} (for meet-semilattices) and \emph{maps preserving arbitrary meets} are similar. 

We set $L\leq_{\vee} L'$, resp. $L\leq_{\nabla} L'$, if there is an embedding of $L$ into $L'$ preserving finite, resp. arbitrary, joins. 

There is a correspondence  between maps from a complete lattice $L$ to a complete lattice $L'$  preserving arbitrary joins and maps from $L'$ to $L$ preserving arbitrary meets. It goes back to Ore. We illustrate it with Theorem \ref{duality}, Theorem \ref{thm:algebraicconvex} and Theorem \ref{thm:compact}  below. 

For this purpose, we recall that a subset $A$ of a complete lattice $L$ is \emph{meet-dense} if every $x\in L$ is the meet (possibly infinite) of elements  of $A$.  The complete \emph{join-dimension} of a complete lattice $L$ is the least cardinal $\kappa$ such that $L$ can be embedded  into a direct product of $\kappa$ complete chains by a map preserving complete joins; we denote it by $dim_{\nabla}(L)$.  The   \emph{join-dimension} of a join-semilattice $L$ is the least cardinal $\kappa$ such that $L$ can be embedded  into a product of $\kappa$  chains by a join- preserving map. We denote it by $dim_\vee(L)$. 

Note that a map $f$ from $L$ to $L'$ which  preserves arbitrary joins sends the least element of $L$ onto the least element of $L'$. Thus, if $L'$ is a product of complete chains $(C_i)_{i<\kappa}$ each with least element $0_i$ then $f(0) =(0_i)_{i<\kappa}$. This equality  is no longer true if $f$ is only join-preserving, but if we consider join-semilattices with a least element, we may impose that this equality holds and the value of the dimension will be the same. 

Note that the join-dimension and complete-join-dimensions may differ from the order-dimension. 

\begin{exm}
\end{exm}
The lattice $M_3$ has  order dimension 2. Indeed, if $a,b,c$ are atoms, then we may  embed it  into a product of two chains: $C_1=0_1<a_1<b_1<c_1<1_1$, and $C_2=0_2<c_2<b_2<a_2<1_2$, and $f(x)=(x_1,x_2)$. While this embedding preserves the order, it does not preserve the join operation. It is easy to check that any attempt at using two chains will not result in $f(a)\vee f(b)=f(a)\vee f(c)=f(b)\vee f(c)$. On the other hand, one can make representation with three chains: $C_x=0_x<x<1_x$, $x=a,b,c$, and $f(x)[y]=x$, if $0<x=y<1$, $1_y$ otherwise. Thus, $dim_\vee(M_3)=3$. 

\begin{exm}
\end{exm}
A more involved example is $\powerset^{<\omega}(\mathbb N)$. Indeed, according to \cite{kierstead-milner} if $E$ is an infinite set of cardinal $\kappa$,  $dim(\powerset^{<\omega}(E))= log_2(log_2(\kappa))$, while it can be shown that  $dim_{\vee}(\powerset^{<\omega}(E))= log_2(\kappa)$.
\begin{exm}
\end{exm}
Let $L:= Co(X,V)$. As already mentioned in inequalities (\ref{dimension}),    
 if $X$ can be covered by $n$ lines ($n<\omega$) then $dim (L) \leq 2\cdot n$. On an other hand, if $X$ is the union of three lines,  $dim_{\vee}(L)$ is infinite. Indeed, according to 
Corollary \ref{cor:dimension} it suffices to observe that  $L$ contains  infinite antichains of completely meet-irreducibles. This fact is easy to prove: for sake of simplicity, we suppose $V:=\R^2$ and $X$ be the union of three lines $\ell, \ell'$ and $\ell''$. Pick $x\in \ell$ such that some line $\delta$ containing $x$ meets  $ \ell'$ in $a'$,  $\ell''$ in $a''$ with   $x$ belonging  to the segment joining $a'$ and $a''$.  Let $\Delta$ be the union of one of the two open half-planes determined by $\delta$ and one of the two open half-lines determined by $\delta$ and $x$. This set is a completely meet-irreducible convex subset of $\R^2$.  It turns out that $\Delta\cap X$ is a maximal convex subset in $L$ not containing $x$ hence is completely meet-irreducible  in $L$. Since distinct  lines $\delta$ provide incomparable completely meet-irreducibles  in $L$, there are  infinite antichains of completely meet-irreducibles.  
%

The \emph{chain-covering-number} of a poset $P$, denoted by $cov(P)$ is the least cardinal number $\kappa$ such that $P$ can be covered by $\kappa$ chains. We recall Dilworth's theorem (see, for example \cite{Nat}): \emph{If $P$ is a (possibly, infinite) poset of finite width $n$, then $cov(P)=n$}.

The following result illustrates the correspondence  we alluded. 


\begin{thm}\label{duality}
For every complete lattice $L$, the following property holds:
\begin{equation}\label{eq:join-representation1}
dim_{\nabla}(L)= Min \{cov(A): A  \; \text {is meet-dense in}\; L\}.
\end{equation}
\end{thm}

\begin{proof}Let $A$ be a meet-dense subset of $L$ and $\mathcal C:= (C_i)_{i\in I}$ be a family of chains whose union covers $A$. Extend each $C_i$ to a maximal chain $L_i$ of $L$ (maximal w.r.t. inclusion). Let $L'$ be the direct product of the $L_i$'s. For each $x\in L$ we  set $f(x):= (\ell_{i}(x))_{i\in I}$ where $\ell_{i}(x):= \bigwedge (L_i \cap \uparrow x)$. This defines an embedding $f$ from $L$ into $L'$ which preserves arbitrary joins.  Indeed, since $L_i$ is maximal in $L$, $L_i$ is a complete sublattice of $L$; in particular  for each $x\in L$, $\ell_{i}(x)$ is the least element of $L_i$ above $x$. Thus the map $\ell_i$ is a retraction of $L$ onto $L_i$ which preserves arbitrary joins. It follows that $f$ preserves arbitrary joins. Since $A$ is meet-dense, $f$  is one-to-one and thus is an embedding. This proves that 
$dim_{\nabla}(L)\leq Min \{cov(A): A  \; \text {is meet-dense in}\; L\}$. Conversely, suppose that there is an embedding $f$ preserving arbitrary joins from  $L$ into a product $L':= \Pi_{i\in I} L'_{i}$  of complete chains.  For $x'\in L'$ set $f_{\vee}(x')= \bigvee f^{-1}( \downarrow x')$. Since $f$ preserves arbitrary joins,  
 $f_{\vee}(x')$ is the largest element $y\in L$ such that $f(y)\leq x$. It follows that the map $f_{\vee}$ preserves arbitrary meets. Since $f$ is one-to-one, it follows that $f_{\vee}\circ f (x)=x$ for every $x\in L$, hence the image $A:= f_{\vee}(A')$ by $f_{\vee}$ of every meet-dense subset $A'$ of $L'$ is meet-dense in $L$.  To conclude, take $A':= \bigcup_{j\in I} A'_j$ where $A'_j:= \{(x_i)_{i\in I}: x_i=1_i \; \text{for all}\;  i\not =j\}$ and $1_i$ is the largest element of $L_i$. Then $A'$ is meet-dense and the union of $\vert I\vert$ chains. The converse inequality follows.  
\end{proof}
\begin{cor}\label{cor:dimension}
If $L$ is an algebraic lattice then: \begin{equation}
dim_{\nabla}(L)= cov(\mi_{\Delta}(L)).
\end{equation}

In particular $dim_{\nabla}(L)= n$ with $n<\infty$ if and only if $n$ is the maximum size of antichains of $\mi_{\Delta}(L)$. 
\end{cor}
\begin{proof}
If $L$ is an algebraic lattice then $\mi_{\Delta}(L)$ is meet-dense and every meet-dense subset of $L$ contains $\mi_{\Delta}(L)$. Apply Equality (\ref{eq:join-representation1}). According to Dilworth's theorem if $n$ is the maximum size of antichains of  a poset $P$ then the covering number of $P$ is $n$. 

\end{proof}
\begin{prop}\label{prop:finite-dimension}Let $P$ be a join-semilattice with a least element. Then $dim_{\vee}(P)\leq dim_{\nabla}(\ideal  P)$. Equality holds if $dim_{\vee}(P)$ is finite. 
\end{prop}

\begin{proof}
For each $x\in P$ set $i(x):= \downarrow x$. The map $i$ is a  join-preserving  embedding of $P$ into $\ideal  P$. The inequality  $dim_{\vee}(P)\leq dim_{\nabla}(\ideal  P)$ follows. 
Let $Q$ be a join-semilattice with a least element and $f$ a join-preserving embedding from $P$ to $Q$ which carries the least element of  $P$ on the least element of $Q$. Set $\overline f(I):= \downarrow f(I)$. This map is an embedding of $\ideal  P$ into $\ideal  Q$ which preserves arbitrary joins. Take for  $Q$ a finite product $\Pi_{i\in I} C_i$ of chains, with $\vert I\vert =dim_{\vee}(P)$. According to Lemma \ref{lm:pdscattered} $\ideal  (\Pi_{i\in I} C_i)$ is isomorphic to $\Pi_{i\in I} \ideal  C_i$.  Hence $ dim_{\nabla}(\ideal  P)\leq \vert I\vert$, as claimed. 
\end{proof}

We illustrate the  proof of Theorem \ref{duality} in the case of algebraic lattices and then convex geometries. This allows  to precise Proposition \ref{prop:finite-dimension}.


We find convenient to view  an algebraic lattice $L$ as the lattice $C:=Cl(X, \phi)$ of an  finitary closure system $(X, \phi)$ such that $\emptyset \in C$. 

 
    Let $x,y\in X$. We set $x\leq y$ if $\phi(x) \subseteq \phi(y)$.  Let $D$ be a  chain  included in $C$. Let $x,y\in X$. We set $x\leq_D y$ if  every $I\in D$ containing $y$ contains $x$. Trivially, 
the relation $\leq_D$ is a total quasi-order on $X$ which contains the order $\leq$ and every $A\in D$ is an initial segment of $(X, \leq_D)$. For every subset $A$ of $X$ we set $\downarrow A:=\{y\in X: y\leq x \; \text {for some}\; x\in A \}$ and $\downarrow_{D} A:= \{y\in X: y\leq_D x \; \text {for some}\;  x\in A\}$. 
\begin{lm}\label{Ci}
Let  $D$ be a maximal chain of the lattice $C$ of an  algebraic closure $(X, \phi)$  then:
\begin{enumerate}
\item  $D=I(X, \leq_D)$, the lattice of initial segments of the quasi-ordered set $(X, \leq_D)$. 

\item  For every  $A\subseteq X$, $\downarrow_{D} A$ is the least member of $D$ containing $A$. 

\item  The map $\ell_D: C\rightarrow D$ defined by setting $\ell_D(A):=\downarrow_{D} A$  preserves arbitrary joins and fixes $D$ pointwise.

\item $\ell_D (A)$ is compact in $D$ for every compact $A$ of $C$. 
\end{enumerate}
\end{lm}

\begin{proof}
$(1)$. This is well known. For reader's convenience we give  a proof. As mentioned above, $D\subseteq I(X, \leq_D)$. Let $I\in I(X, \leq_D)$. Since $D$ is maximal, $\emptyset$ and $X$ belong to $D$, hence we may suppose $I$ distinct of $\emptyset$ and $X$. Let $y\not \in I$ and let $x\in I$ such that $x<_Dy$ (that is  $x\leq_D y$ and $y \not \leq_D x$). By definition of $\leq _D$ there is some $I_{x,y}\in D$ such that $x\in I_{x, y}$ and $y\not \in I_{x,y}$. Since $D$ is a maximal chain of $C$ and $Cl(X, \phi)$ is algebraic, $D$ is closed under union. Hence, $I_y:= \bigcup_{\{x\in I  : x<_D y\}} I_{x,y}$  belongs to $D$.  Again, since $D$ is a maximal chain of $C$, $D$ is closed under intersection.  Hence $I=\bigcap_{y\not \in I} I_y \in D$.  
$(2)$. Follows immediately from $(1)$. 
$(3)$. The fact that the map $\ell_D$ preserves arbitrary joins  
 was already indicated in  the proof of Theorem \ref{duality}. 
 $(4)$. Let $A\in K(C)$. Let $F$ be a finite subset of $A$ such that $\phi(F)=A$. Since $\phi(F)$ and $\ell_D(F)$ are the least member of $C$, respectively of $D$,   containing $F$, we have $\phi(F) \subseteq \ell(D)$. This yields $\ell_D(\phi(F))\subseteq \ell_D(\ell_D(F))$, that is $\ell_D(A)\subseteq \ell_D(F)$. Since $F\subseteq \ell_D(A)$, we obtain $\ell_D(A)=\ell_D(F)$. Hence $\ell_D(A)\in K(D)$. 
 
 \end{proof}

Let $C$ be a set; we set $C_{*}:=C\setminus \{\emptyset\}$ and $\hat C:=C\cup \{\emptyset\}$. 
Let $\{C_k:k\in K\}$ be  a family of closure systems  $C_k$ on the same set $X$.  The \emph{join} of this family denoted by $\bigvee_{k\in K} C_k$ is the closure system on $E$ with closed sets $A:=\bigcap_{k\in K} A_k$, where $A_k\in C_k$ for each $k\in K$. We give  in $(4)$ of Lemma \ref{join} below a  presentation of the join-semilattice $K(C)$ of compact elements of a closure $C$ when $C$ is the join of finitely many chains $C_k$, $k<n$, and the $C_k$'s are maximal in $C$.

\begin{lm}\label{join}
Let $C:=Cl(X, \phi)$,  where $\phi$ is algebraic,    $M:=\mi_{\nabla}(C)$ and $(C_k)_{k\in K}$ be   a family of chains of $C$. Then, the union of the $C_k$'s   covers $M$ iff 
$C= \bigvee_{k\in K} C_k$,  the  join of  the closures $C_k$'s. Furthermore, if one of these conditions holds and the $C_k$ are maximal in $C$ then: 
\begin{enumerate}
\item The quasi-order $\leq$ is the intersection of the total quasi-orders $\leq_{C_k}$'s; 
\item For every 
 $I\in C$, $I= \bigcap_{k\in K} \ell_{k}(I)$ where $\ell_{k}(A):= \downarrow_{C_k}A$ for every $A\subseteq X$;  and  
 \item The map $\hat \ell$ defined by setting  $\hat \ell(A):= (\ell_{k}(A))_{k\in K}$  for every non-empty $A\subseteq X$ and $\hat\ell (\emptyset):= \emptyset$ induces a one to one map from  $C$ into  $\widehat {\Pi_{k\in K}C_{k *}}$ which preserves arbitrary joins. 
\item If $K$ is finite, $K:= \{0, \dots n-1\}$,  then $\hat\ell$ induces  a one-to one join-preserving map from $K(C)$ into $\widehat {\Pi_{k<n} K(C_k)_{*}}$. 
Furthermore, the image of $K(C)_{*}$ is the join-semilattice of $\Pi_{k<n} K(C_k)_{*}$ generated by the image of the set $K(C)_{1}:= \{\phi(x): x\in X\}$. 

\end{enumerate}
\end{lm}
\begin{proof} Since $C$ is algebraic, $M$ is meet-dense. Hence, if the union of the $C_k$'s   covers $M$, then  $C= \bigvee_{k\in K} C_k$. Conversely, let $I\in M$. Since  $C= \bigvee_{k\in K} C_k$,  $\bigcup_{k\in K} C_k$ is meet-dense;  since $I$ is completely meet-irreducible, $I\in \bigcup_{k\in K} C_k$.  

$(1)$. The intersection of the quasi-orders $\leq_{C_k}$ contains $\leq$. Let $x,y\in X$ such that $x\not \leq y$.  By definition of $\leq$ there is some $I\in C$ such that $x\not \in I $ and $y\in I$. If  $I\in \bigvee_{k\in K} C_k$, there is  some $C_k$ and some $I_k\in C_k$ such that $x\not \in I_k $ and $I\subseteq I_k$, hence $x\not \leq_{C_k}y$.

 $(2)$ By Lemma \ref{Ci}, $\ell_{k}(I)$ is the least member of $C_k$ containing $I$. Since $C= \bigvee_{k\in K} C_k$,  we have $I= \bigcap_{k\in K} \ell_{k}(I)$. 
 
 $(3)$. As indicated in the proof of Theorem \ref{duality},  each map $\ell_k$ preserves arbitrary joins. Hence, the map $\ell := (\ell_k)_{k<n}$ from $C$ into the direct product $\Pi_{k\in K}C_{k}$  preserves arbitrary joins. 
Due to $(2)$ above,  the map $\ell$ is one to one. Since $\phi(I)\not=\emptyset$ amounts to $\ell_k(I)\not=\emptyset$ for every $k<n$, $\ell$ maps  $C_{*}$ into ${\Pi_{i<n}C_k)_{*}}$, hence $\hat\ell$ has the property stated. 

$(4)$. According to $(3)$ and $(4)$ of Lemma \ref{Ci}, $\ell_k$ induces a join-preserving map from $K(C)$ onto $K(C_k)$. According to $(3)$ above  $\ell$ is one-to-one and join-preserving, hence $\hat\ell$ has the property stated. Finally, observe that  $K(C)_{*}$ is generated as a  join-semilattice by $K(C)_{1}$; since  $\ell$ is join-preserving, the image of $K(C)_{*}$ is the join-semilattice of $\Pi_{k<n} K(C_k)_{*}$ generated by the image of $K(C)_{1}$.
\end{proof}
 
 We describe in Theorem \ref{thm:algebraicconvex} below algebraic convex geometries with finite join-dimension. 
 
 We start with three simple lemmas. 
 \begin{lm}\label{lm:convex}
If $C_k$, $k \in K$, are convex geometries, then $C:=\bigvee_{k\in K} C_k$ is a convex geometry as well. \end{lm}

\begin{proof} Let $D$ be  a maximal chain of $C$ and $\phi_D$ the corresponding closure (namely $\phi_D(A)$ is the least member of $D$ containing $A$). Let $Y\subseteq X$ and $x\not= y  \in X\setminus \phi_D(Y)$. 
$(b)$.  The closure operator associated with $C$ is defined as $\phi(Y):=\bigcap_{k\in K} \phi_k(Y)$, for any $Y\subseteq X$, where $\phi_k$ is the closure operator associated with $C_k$.
It is enough to show that $\phi$ satisfies the anti-exchange axiom. Take $x\not = y$, so that $x,y
\not \in X=\phi(X)$ and $x \in \phi(X\cup \{y\})$. Then $x \not \in \phi_k(X)$, for some $k \in K$, but $x \in \phi_k(X\cup \{y\})\subseteq \phi_k(\phi_k(X)\cup \{y\})$. Since $\phi_k$ satisfies the anti-exchange axiom, we have $y \not \in 
\phi_k(\phi_k(X)\cup \{x\})$. Then $y \not \in \phi_k(X\cup \{x\})$, thus $y \not \in \phi(X\cup \{x\})$.
\end{proof}

\begin{lm}\label{lm: linearconvex}
If $C$ is an algebraic  convex geometry then every maximal chain $D$ is a convex geometry, in fact $D= I(X, \leq_D)$ for some linear order $\leq_D$ on $X$. \end{lm}

\begin{proof} By $(1)$ of Lemma \ref{Ci}, we have  $C_k= I(X, \leq_k)$ for some  total quasi-order $\leq_k$. Let $x\not = y\in X$ such that $x\leq_D y$ and $y\leq_Dx$. Since $C$ is algebraic, there is a largest member $A$ of $C_k$ not containing $x$; in fact $A=\{z\in X: z\leq_D x \; \text{and} x\not \leq_D z\}$. Since $D$ is maximal, $A':=\downarrow_D x$ is a cover of $A$ in $D$. Since $A\subset  \phi(A\cup \{x\}) \subseteq A'$  we have $\phi(A\cup \{x\})=A'$ Similarly, we have $\phi(A\cup \{y\})= A'$. This contradicts the fact that $C$ is a convex geometry. \end{proof}

\begin{lm}\label{lm:algebraic}
Suppose that  $K$ is finite. If   each $C_k$ is algebraic, then $C:=\bigvee_{k\in K} C_k$ is algebraic.
\end{lm}

\begin{df}
A \emph{multichain} is a relational structure  $M:=(X, (\mathcal L_k)_{k=1, n})$ where each $\mathcal L_k$ is a linear order  on the set $X$. If $n=2$ this is a \emph{bichain}. The \emph{components} of $M$ are the chains $C_k:= (X, \mathcal L_k)$ for $k:=1, n$. 
\end{df}
 
We denote by $C(M)$ the join $\bigvee_{i=1,n}I(C_k)$. 

For finite convex geometries, see this result in Edelman and Jamison \cite{EdJa}. It is not generally true in case of infinite dimension, see Wahl \cite{W01} and Adaricheva \cite{A14}.

\begin{thm}\label {thm:algebraicconvex}
A closure system $(X, \phi)$ is an algebraic convex geometry with finite dimension iff $Cl(X,\phi)=C(M)$ for some multichain $M$. 
\end{thm}

\begin{proof}
Let  $M:= (X, (\mathcal L_k)_{k=1, n})$. For each $k:=1, n$, the set $D_k:=I(C_k)$  of initial segments of $C_k$ forms an algebraic convex geometry. Hence, from  Lemma \ref{lm:convex} and Lemma \ref{lm:algebraic},   $C(M)$ is an algebraic convex geometry of dimension a most $n$. Conversely, if $(X, \phi)$ is an algebraic convex geometry of dimension at most $n$, then there is  a family  $(D_k)_{k=1, n}$ of maximal chains of $Cl(X,\phi)$ such that $Cl(X,\phi)= \bigvee_{k=1, n} D_k$. By  Lemma  \ref {lm: linearconvex},  $D_k=I(X, \leq_k)$  for some linear order $ \leq_k$ on $X$, hence $Cl(X,\phi)= C(M)$ for $M:= (X, (\leq_k)_{k=1, n})$. 
\end{proof}

Let $M:=(X, (\mathcal L_k)_{k=1, n})$ be a multichain, $C_k:= (X, \mathcal L_k)$ be its components  and $L:=\Pi_{i <n}C_k$. Let $\delta : X \rightarrow L$, be \emph{the diagonal mapping} defined by  $\delta(x)=\langle x,\dots, x)\rangle$, $\delta (X)$ be its range, let $\Delta(L)$ be the join-semilattice of $L$ generated by $\delta(X)$ and $\hat \Delta(L)$ be the join-semilattice obtained by adding a least element, say $\{0\}$, to $\Delta(L)$.

\begin{thm}\label{thm:compact} Let $M$ be  a multichain and   $C:= C(M)$ then $K(C)$ is isomorphic to $\hat \Delta(L)$ via a join-preserving map.
\end{thm}
\begin{proof} According to $(4)$ of Lemma \ref{join},  $\hat\ell$ induces  a one-to one join-preserving map from $K(C)$ onto the  join-semilattice of $\Pi_{k<n} K(C_k)_{*}$ generated by the image of the set $K(C)_{1}:= \{\phi(x): x\in X\}$ augmented of a least element. To conclude, observe that each $K(C_k)_{*}$ is equal to $K(C_k)_1$ which is isomorphic to $C_k$. 
\end{proof}
 
In the next section, we look at the case of bichains.

\section{The semilattice $\Omega(\eta)$ as an obstruction in algebraic convex geometries}\label{omega}

As it was mentioned in the introduction, the semilattice $\Omega(\eta)$ does not appear in the semilattice  of compact elements of an algebraic modular lattice, see \cite{ChaPou1}.
The goal of this section is to demonstrate with  Theorem \ref{mainduplex} that $\Omega(\eta)$ is a typical subsemilattice of compact elements  of convex geometries associated to bichains (see Theorem \ref{mainduplex}).

A \emph{bichain} is a relational structure  $B:=(X,\leq_1,\leq_2)$ where  $\leq_1$ and leq  $\leq_2$ are two linear orders on the set $X$.  To $B$  we  associate its \emph{components} $C_1:= (X, \leq_1)$, $C_2:= (X, \leq_2)$, the lattice $L:= C_1\times C_2$,  the convex geometry $C(B):= I(C_1)\vee I(C_2)$ and the join-semilattice $K(C(B))$ of its compact elements. According to Theorem \ref{thm:compact}, $K(C(B))$ is isomorphic to $\hat \Delta(L)$. We give an other presentation of this lattice. 

Let $L$ be the direct product of two chains $C_1:= (X_1, \leq_1)$ and $C_2:= (X_2, \leq_2)$. Suppose that $X_1$ and $X_2$ have the same cardinality and  let $f:X_1\rightarrow X_2$ be a bijective map from $X_1$ onto $X_2$. Let $\delta_f: X_1\rightarrow  X_1\times X_2$ be defined by $\delta_f(x):= (x, f(x))$ and $\Delta(L, f)$ be the join semilatice of $L$ generated by $\delta_f(X_1)$. Let $\leq_f$ be the inverse image of the order $\leq_2$ by $f$  that is $u\leq_f v$ whenever $f(u)\leq_2f(v)$; let $\delta: X_1\rightarrow X_1\times X_1$ and $L_f:=(X_1, \leq _1) \times (X_1, \leq_f)$ and let $\Delta(L_f)$ be the join semilatice of $L_f$ generated by $\delta(X_1)$.

\begin{lm}\label{duplexisomorphism}  $(a)$ $\Delta(L, f)=\{(x', f(x''))\in X_1\times X_2: x''\leq_1x' \; \text{and} \; f(x')\leq_2 f(x'')\}$.  $(b)$ The join-semilattices $\Delta(L, f)$ and $\Delta(L_f)$ are isomorphic. In particular, with a bottom element added,  $\Delta(L,f)$ is isomorphic to $K(C(B))$ where $B=(X_1, \leq_1, \leq_f)$. 
\end{lm} 
\begin{proof} $(a)$ Let $Z:= \{(x', f(x''))\in X_1\times X_2: x''\leq_1x' \; \text{and} \; f(x')\leq_2 f(x'')\}$. 
We prove that $Z=\Delta(L, f)$. First, $Z$ is a join-semilattice of $L$. Indeed, let $(x', f(x''))$ and $(y', f(y''))$ in $Z$. Let $u:=(u',u'')$ be their join. We claim that $u\in Z$. Indeed, $u'=Max_{C_1} (\{ x',y'\})$ and $u''=Max_{C_2} (\{ f(x''), f(y'')\})$. W.l.o.g. we may suppose $y'\leq_1x'$ that is $u'= x'$. If $u''=f(x'')$ then $u=(x', f(x''))$ hence $u\in Z$ as required. Otherwise $u''= f(y'')$. Since $u'=x'$ we have $y'\leq_1 x'$ and since $(y', f(y''))\in Z$ we have $y''\leq_1 y' $ hence $y''\leq_1x'$. Since 
$u''= f(y'')$ we have  $f(x'')\leq_2 f(y'')$ and since $(x', f(x''))\in Z$ we have $f(x')\leq_2 f(x'')$ hence $f(x')\leq f(y'')$. It follows that $u=(x', f(y''))\in Z$ as required. Next, we have trivially $\delta_f(X_1)\subseteq Z$. Since $Z$ is a join-semilattice, it follows that $\Delta(L, f)\subseteq Z$. Finally, we have $(x', f(x''))= (x', f(x'))\vee (x'', f(x''))$ whenever $(x',f(x''))\in Z$ hence $Z\subseteq \Delta(L, f)$. The equality $Z=\Delta(L, f)$ follows. 

$(b)$ Let $g:X_1\times X_1 \rightarrow X_1\times X_2$ defined by setting $g(x,y):= (x, f(y))$. As it is easy to check, $g$ is an isomorphism of $L_f$ onto $L$. Since it carries $\delta(X_1)$ onto $\delta_f(X_1)$, it carries  $\Delta(L_f )$ onto $\Delta(L, f)$. \end{proof}

 If the order-type  of the first component  of a bichain $B$ is   $\omega$ and the second component is  non-scattered, we say that the lattice $C(B)$ is a  \emph{duplex}, for convenience. We denote by $ \mathcal {L_D}$ the class of join-semilattices isomorphic to $K(C)$ for some duplex $C$. 

Two  join-semilattice are  \emph{equimorphic} as join-semilattices if each one is embeddable into the other by some join-preserving map. We have:

\begin{prop}\label{prop:omegaeta} Let $S'\in  \mathcal {L_D}$.   A join-semilattice $S$ is equimorphic to $S'$ if and only if $S$ is embeddable as a join-semilattice in the  product  $L$ of two chains $C_1:= (X_1, \leq_1)$ and $C_2:= (X_2, \leq_2)$  
in such a way that:
\begin{enumerate}
\item the first projection $A_1$ of $S$ has order type $\omega$,  
\item the second projection $A_2$ of $S$ is non-scattered, 
\item the set $S(x):=(\{x\}\times A_2 )\cap S$ is finite  for every $x\in A_1$. 
\end{enumerate}
\end{prop}
\begin{proof}
Let us prove first that every $S\in \mathcal {L_D}$ satisfies the conditions above.  Let $B:=(X,\leq_1,\leq_2)$ be a bichain, $C_1:= (X, \leq_1)$, $C_2:= (X, \leq_2)$, $L:= C_1\times C_2$,  $C(B):= I(C_1)\vee I(C_2)$ and $S:=K(C(B))$ be the join semilattice of its compact elements. According to Theorem \ref{thm:compact}, $S$ is isomorphic to $\hat \Delta(L)$. Add to $X$ a new element $a$, set $X':= X\cup \{a\}$,  extend both orders to $X'$, deciding that  $a$ is the least element of $X'$ w.r.t. each order. Let $C'_1:= 1+C_1$ and  $C'_2:= 1+C_2$ be the resulting chains, let $L':= C'_1\times C'_2$, let $\delta' (x'):= (x',x')$.  Clearly $\hat \Delta(L)$ is isomorphic to $\Delta(L')$, the join-semilattice of $L'$ generated by $\delta'(X')$. Hence we may suppose that $S=\Delta(L')$. Supposing that $C(B)$ is a duplex,  the projections $A'_1$ and $A'_2$ of $S$ satisfy conditions $(1)$ and $(2)$. Let us prove that condition $(3)$ holds. Let $x\in A'_1$. According to $(a)$ of Lemma \ref{duplexisomorphism}, $\Delta(L')=\{(x', x'')\in X' \times X': x''\leq_1x' \; \text{and} \; x'\leq_2 x''\}$, hence $S(x)= \{(x, x'')\in X'\times X': x''\leq_1x \; \text{and} \; x\leq_2 x''\}$. Since $A'_1$ has order type $\omega$, every proper initial segment is finite, hence $S(x)$ is finite.

Now,  suppose that $S$ is equimorphic to $\hat \Delta(L)$. Let $L'$ such that  $\hat \Delta(L)$ is isomorphic to $S':=\Delta(L')$. We may suppose that $S$ is a join-semilattice of $S'$ with projections $A_1$ and $A_2$. Since $S'$ satisfies conditions $(1)$ and $(3)$, $S$ satisfies these conditions.  Now, $\ideal S'$ is non-scattered, indeed, the second projection $A'_2$ of $S'$ embeds into $\ideal S'$ (associate $(A_1\times \downarrow_{L{_{2}}} {y})\cap S'$ to every $y\in A'_2$). Since $S$ is equimorphic to $S'$, $\ideal S'$ embeds into $\ideal S$, hence $\ideal S$  is non-scattered. According to Lemma \ref{lm:pdscattered} and Lemma \ref {q in product}, $A_1$ or $A_2$ is non-scattered. Since $A_1$ has order-type $\omega$, $A_2$ is non-scattered, hence condition $(2)$ holds. 

Conversely, suppose that $S$ satisfies the three conditions stated above. 
\begin{claim} \label{claim:horizontalline}If  $A'_2$ is a countable subset of $A_2$ such that  $C_2{\restriction A'_2}$ has order-type $\eta$ then the set $ (A_1\times \{y\} )\cap S$ is infinite  for every $y\in A_2'$. 
\end{claim}
\noindent {\bf Proof of Claim \ref{claim:horizontalline}}. 
Enough to prove that for every  $x_1\in A_1$ and $y\in A_2'$ there is some $x, x_1\leq_1x$ such that $(x,y)\in S$.  Let $(x_1, y)\in A_1\times A_2'$. For every $y'\in A_2'$  there is some $t'$ such that $(t',y')\in S$. Consider $t$ such that $(t,y)\in S$ and $t_1:=Max_{C_1} \{x_1, t\}$ . From our hypothesis, the set of $y'\in A_2'$ such that $y'<y$ is infinite and the set $((A_1\cap \downarrow t_1)\times X_2)\cap S$ is finite. Hence,  there is some  $y'\in A_2'$ with $y'<_2 y$   and some $x', t_1<_1x'$ such  that $(x',y')\in S$. Since $S$ is a join-semilattice,  $(t,y)\vee (x',y')\in S$. Since $(t,y)\vee (x' ,y')= (x', y)$ we may set $x=x'$ proving our claim. \hfill $\Box$

\begin{claim} \label{claim:subsemilattice} $S$ contains  some  join-semilattice $S_1\in  \mathcal {L_D}$. 
\end{claim}
\noindent {\bf Proof of Claim \ref{claim:subsemilattice}}. 
Let $A'_2$ be a countable subset of $A_2$ such that  $C_2{\restriction A_2'}$ has order-type $\eta$. We prove that there is a one-to-one map $g$ from $A_2'$ into $A_1$ such that  $(g(y), y) \in S$ for  each $y\in A_2'$. Indeed, enumerate the elements of $A_2'$ into a sequence $y_0, \dots  y_n, \dots $;  pick $x_0$ arbitrary in $A_2'$ such that $(x_0, y_0)\in S$ and set $g(y_0):=x_0$ . Suppose $g(y_m)=x_m$ be defined for $m<n$. According to Claim \ref{claim:horizontalline}, there is  some $x$ larger than all $x_m$'s such that $(x, y_n) \in S$, set $g(y_n): =x_n:=x$. Let $S_1$ be the join-subsemilattice of $C_1\times C_2$ generated by the set $\{(g(y), y): y\in A_2'\}$.  Setting $A_1':= g(A_2')$ and $f:= g^{-1}\Delta(L')$ we may apply Lemma \ref{duplexisomorphism}, hence $S_1\in  \mathcal {L_D}$.
\hfill $\Box$

\begin{claim} \label{claim:supersemilattice}$S$ embeds by a join-preserving map into  $S'$.
\end{claim}
\noindent {\bf Proof of Claim \ref{claim:supersemilattice}}. 
We may suppose that $A_1=X_1$ and $A_2=X_2$. 
Let $L'$ such that $S'$ is isomorphic to $\hat \Delta(L')$, where $L=C'_1\times C'_2$, $C'_1:=(X',  \leq'_1)$ is a chain of order-type $\omega$ and $C'_2:=(X', \leq'_2)$ is a non-scattered chain.  We are going to define a   map $F$ from $L$ into $L'$ such that $F(S)\subseteq \Delta(L')$. The map $F$ will be of the form $F(x, y):= (f (x), g(x))$, with $f$ and $g$ embeddings of $C_1$ into $C'_1$ and of $C_2$ into $C'_2$. It will follow that $F$ will be one-to-one and will preserve joins and meets, hence its restriction to $S$ will be a join-embedding  into  $\Delta(L')$, hence, into $\hat \Delta(L')$. First, we define $g$. Let $A''_2$ be a subset of $A'_2$ such that $C'_2{\restriction A''_2}$ has order type $\eta$. Due to condition $(3)$,  $A_2$ is countable, hence, from  a Cantor's result, $C_2$ is embeddable into $C'_2{\restriction A''_2}$. Let $g$ be such an embedding. Next, we define $f$. We proceed by induction. Since $C_1$ has order-type $\omega$, we may enumerate the elements of $A_1$ into the sequence $a_0<_{C_1}<a_1<_{C_1}\dots a_n<_{C_1}\dots $. Let $n\in \N$  and $A_1(n):=\{a_m:m<n\}$. Suppose  that $f$ defines an embedding of   $C_1{\restriction{A_n}}$ into $C'_1$ and $F((A_1(n)\times A_2)\cap S)\subseteq \Delta(L')$. We extend $f$ to $A_{n+1}$ in such a way that $(f(a_n), g(y))\in  \Delta(L')$ for every $y\in S(a_n)$. Doing so we will get $f$ as required.  For that,  observe that since $S'\in  \mathcal {L_D}$,  Claim \ref{claim:horizontalline} applies. Hence, $(A'_1\times \{y'\})\cap S'$ is infinite for every $y'\in A''_2$, thus for every $y\in S(a_n)$  there is some $a'_y\in A'_1$ with $f(a_{n-1}) <_{C'_1} a'_y$ such that $(a'_y, g(y))\in S'$. According to condition $(3)$, $S(a_n)$ is finite, hence we may set $y_0:= Min_{C_2} S(a_n)$ and $y'_0:= g(y_0)$. Since $y'_0\in A''_2$, $(A'_1\times \{y_0'\})\cap S'$ is infinite, hence there is some  $a'\geq_{C'_1} Max_{C'_1}\{ a'_y: y\in S(a_n)\}$ such that $(a', y'_0)\in S'$. Since $S'$ is a join-semilattice, we have $(a', y'_0)\vee(a'_y, g(y))= (a', g(y))\in S'$  for every $y \in S(a_n)$. It suffices to set $f(a_n):=a'$.

\hfill $\Box$ 

From Claim \ref{claim:subsemilattice} we have $S_1\leq_{\vee} S$. Since $S'$ satisfies conditions $(1)$, $(2)$, $(3)$ and $S_1\in \mathcal {L_D}$, Claim \ref {claim:supersemilattice} asserts that  $S'\leq_{\vee} S_1$. 
 From Claim \ref{claim:supersemilattice} we have $S\leq_{\vee} S'$. Hence, $S$ and $S'$ are equimorphic as join-semilattices. 
\end{proof}

%

\begin{thm}\label{mainduplex}
$\Omega(\eta)$ is equimorphic to any member of $ \mathcal {L_D}$.
\end{thm}

\begin{proof}
By construction, $\Omega(\eta)$ is a join-semilattice of the product $ C_1\times C_2$,  where $C_1$ is  the chain of non-negative integers and   $C_2$ is the chain of  dyadic numbers of the the interval $[0, 1[$. The second projection  being surjective, its image is non-scattered; by construction, each intersection $(\{x\}\times X_2 )\cap \Omega(\eta)$ is finite, hence Proposition \ref{prop:omegaeta} yields that $\Omega(\eta)$ is equimorphic,   as  a subsemilattice, to  any  $S\in \mathcal {L_D}$. 
\end{proof} 

Note that $\Omega(\eta)$ is not isomorphic to the semilattice of compact elements of an algebraic convexity (It will not satisfy Theorem \ref{thm:algebraicconvexity}).

The fact that members of $ \mathcal {L_D}$ are equimorphic  as join-semilattices  which follows  from Proposition \ref{prop:omegaeta} can be derived from a result of \cite{pouzet-zaguia} as explained below. 

An \emph{embedding} of a bichain $B:= (X,\leq_1,\leq_2)$ into a bichain $B':=  (X', \leq'_1, \leq'_2)$ is any  map $f:X \rightarrow X'$ which is an embedding of $C_1:=(X, \leq_1)$ into $C'_1:=(X', \leq'_1)$ and an embedding of $C_2:=(X, \leq_2)$ into $C'_2:= (X', \leq'_2)$. Two bichains are \emph{equimorphic} if each one is embeddable into the other.

\begin{lm}\label{equimorphy}

Let $B:= (X, \leq_1, \leq_2)$ and $B':= (X', \leq'_1, \leq'_2)$ be two bichains. 

\begin{enumerate}

\item  If $B$ is embeddable into $B'$ then $C(B)$ is embeddable into  $C(B')$ by a map preserving complete joins and $K(C(B))$ is embeddable into $K(C(B')$ by a join-preserving map. 
\item If the first components of $B$ and $B'$ are isomorphic to $\omega$ and the second components are non-scattered,then $B$ and $B'$ are equimorphic. 
\end{enumerate}
\end{lm}
\begin{proof}
Item (1).  Let $\varphi$, resp. $\varphi'$,  be the closure associated with $C(B)$, resp. $C(B')$. Let $f: X\rightarrow X'$ be an embedding of $B$ into $B'$. For $A\in \powerset (X)$, set $\overline f(A):= \varphi'(f(A))$. By definition,   the restriction of $\overline f$ to $C(B)$ maps $C(B)$ to $C(B')$. We check that it preserves arbitrary joins. 
This relies on the fact that $\overline f(A)= \overline f(\varphi (A))$ for every $A\in \powerset (X)$. This fact is easy to check. Let $x'\in \overline f(\varphi (A))$. Then,  there are $x'_1,x'_2\in f(\varphi (A))$ such that $x'\leq_1 x'_1$ and $x'\leq_2 x'_2$. There are  $x_1, x_2\in \varphi (A)$ such that $x'_1= f(x_1)$ and $x'_2= f(x_2)$. Since $x_1\in \varphi(A)$ there is  $y_1\in A$ such that $x_1\leq y_1$; similarly, there is  $y_2\in A$ such that $x_2\leq y_2$. Let $y'_1:= f(y_1)$ and $y'_2:= f(y_2)$; since $f$ preserves the two orders, we have $x'_1\leq'_1 y'_1$ and  $x'_2\leq'_2 y'_2 $; by transitivity we obtain $x'\leq'_1 y'_1$ and $x'\leq'_2 y'_2$. Hence $x'\in \varphi' (f(A))=\overline f(A$. This yields $ \overline f(\varphi (A))\subseteq \overline f(A)$. The reverse inclusion being trivial, we obtain the equality. Let $(A_i)_{i\in I}$ be a family of members of $C(B)$ and $A:= \bigvee_{i\in I}A_i$. We check that $\overline f(A)= \bigvee_{i\in I}\overline f(A_i)$.  We have $A= \varphi( \bigcup _{i\in I}A_i)$, hence $\overline f (A)= \overline f(\varphi ( \bigcup _{i\in I}A_i))= \overline f( \bigcup _{i\in I}A_i)$ by he fact above.  Since, we  have   $\overline f( \bigcup _{i\in I}A_i)= \varphi'( f( \bigcup _{i\in I}A_i))= \varphi' ( f( \bigcup _{i\in I}A_i))= \bigvee_{i\in I}\varphi' ( f(A_i) )=\bigvee_{i\in I}\overline f(A_i)$, we obtain the desired equality. For the last part of the statement, obseve that the map $\overline f$ induces an embedding of $K(C(B))$ into   $K(C(B'))$.
 
Item (2). This is a consequence of Corollary 3.4.2, p.167 of \cite{pouzet-zaguia}. 
\end{proof}

\begin{comments} The notion of bichain is not so peculiar. Several papers related to bichains have appeared during the last few years. Some are about infinite bichains and are mostly concerned by their  endomorphisms (\cite{sauer-zaguia}, \cite {laflamme}, \cite{delhomme-zaguia}). Many are  about finite bichains and originate in the study of classes of permutations. To each permutation $\sigma$ of $[n]:= \{1, \dots, n\}$ one associates  the  bichain  $B_{\sigma}:= ([n], \leq, \leq_{\sigma})$ where $\leq$ is the natural order on $[n]$
and $\leq_{\sigma}$ the linear order defined by
$i\leq_{\sigma} j$
if and only if  $\sigma(i) \leq \sigma(j)$. Conversely, if $B:=(E, C_1,C_2)$ is a finite bichain, then $B$ is isomorphic to a bichain $B_{\sigma}$ for a unique permutation $\sigma$ on $[\vert E\vert~]$. Now, if $\sigma$ and $\pi$ are two permutations with domains $[n]$ and $[m]$,  one can set  $\sigma \leq \pi$ if and only if $B_{\sigma}$ is embeddable into $B_{\pi}$. This defines an order on the class $\mathfrak S$ of all permutations.  Classes $\mathcal C$ of permutations such that $\sigma \in \mathcal C$ whenever $\sigma \leq \pi$ for some $\pi\in \mathcal C$ are called $\emph{hereditary}$. Many results have been devoted to the study of he behavior of the function $\varphi_{\mathcal C}$ which counts for each integer $n$ the number $\varphi_{\mathcal C}(n)$ of permutations $\sigma$ on $n$ elements which belong to $\mathcal C$, see the survey \cite{klazar}. 

The correspondence between permutations and bichains was noted by Cameron \cite{cameron} (who rather associated to  $\sigma$ the bichain  $([n], \leq, \leq_{\sigma^{-1}}))$. It allows to study classes of permutations by means of the theory of relations.
In particular, via this correspondence, hereditary classes of permutations correspond to hereditary classes of bichains and \emph{simple permutations}, a key notion in the study of hereditary classes (see the survey  \cite {brignall} and \cite{nozaki}),  correspond to \emph{indecomposable bichains}, which become, via this correspondance an  important class of  indecomposable structures (see \cite{ehren}).


\end{comments}

\section{Order scattered algebraic lattices  with finite join-dimension}\label{section:orderscattered}

In this section we characterize by obstructions  order scattered algebraic lattices
with finite join-dimension.

The motivation comes  from the following result (\cite {pouzet-zaguia}, Theorem 2, p.161):

\begin{thm} Let $P$ be an ordered set. Then $\ideal  P$  is order-scattered iff $P$ is order-scattered and $\Omega(\eta)$ is not embeddable into $P$. 
\end{thm}

 The following conjecture is stated in \cite{ChaPou1}.

\begin{con}
If $P$ is a join-semilattice, then  $\ideal  P$  is order-scattered iff $P$ is order-scattered, and neither  $\powerset^{<\omega}(\mathbb N)$, nor $\Omega(\eta)$,  is embeddable into $P$ as a join-semilattice.
\end{con}

It must be noticed that while $\Omega(\eta)$ is embeddable into $\powerset^{<\omega}(\mathbb N)$ as a poset, it is not embeddable as a join-semilattice. In fact, as was shown in \cite{ChaPou2} Corollary 1.8 p.4:

\begin{thm} \label{poset} A join-subsemilattice $P$ of $\powerset^{<\omega}(\mathbb N)$ contains either  $\powerset^{<\omega}(\mathbb N)$ as a join-semilattice or is well-quasi-ordered (that is $P$ contains neither an  infinite antichain nor an infinite descending chain). 
\end{thm}

More generally, if the lattice $L: =\ideal  P$ is  modular then  $\Omega(\eta)$ cannot appear as a join-subsemilattice of $P$. In this case, the conjecture above was proved in \cite{ChaPou0}. When $L$ is a convex geometry, as was shown in section \ref{omega}, 
$\Omega(\eta)$ may be a join-subsemilattice of $P$.

We are aiming at proving the conjecture for arbitrary algebraic convex geometries, but for now we restrict the result to convex geometries $L:=\ideal  P$, for which $P$ has a finite $\vee$-dimension.

\begin{thm}\label{scat-variety}
Let $P$ be a semilattice with $dim_\vee(P)=n< \infty$. Then the following properties are equivalent:
\begin{enumerate}[(i)]
\item $P$ is embeddable by a join-preserving map  into a product of $n$ scattered chains.
\item $P$ is embeddable by a join-preserving map  into a product of finitely many scattered chains.
\item $\ideal  P$ is topologically scattered;
\item $\ideal  P$ is order-scattered;
\item  $\mi_{\Delta}(\ideal  P)$ is order-scattered;
\item  $\overline{\mi_{\Delta}(\ideal  P)}$ is topologically scattered.
\end{enumerate}
\end{thm}
\begin{proof} 
\noindent $(i)\Rightarrow (ii)$. Obvious. 

\noindent $(ii)\Rightarrow (iii)$. Let $f$ be a join-embedding of $P$ into a finite product of chains say $Q=\Pi_{k\in K} C_k$. W.l.o.g we may  suppose that $P$ and each $C_k$ have a least element $0$ and $0_k$ respectively, and that  $f(0):= (0_k)_{k\in K}$.  Since $\ideal C_k$ is isomorphic to $I(C_k^*)$ where $C_k^*:= C_k\setminus \{0_k\}$, if  $C_k$ is order-scattered, then by  Proposition  \ref{lem:scattered chain}, $\ideal  C_k$ is topologically scattered. This being true for all $k\in K$, $\Pi_{k\in K} \ideal  C_k$ is topologically scattered  from $(4)$ of Lemma \ref{product}. As noticed in the proof of Proposition \ref {prop:finite-dimension}, $\ideal Q$ is isomorphic to $\Pi_{k\in K} \ideal  C_k$. Hence, in our case, $\ideal  Q$ is topologically scattered.   Since $f$ is join-preserving,   the set $f^{-1}(I)$ belongs to $\ideal  P$ for $I\in \ideal  Q$. Let   $g$ be the map from $\ideal  Q$ to $\ideal  P$ defined by setting $g(I):= f^{-1}(I)$. As it is easy to check, this map  is continuous. Hence by $(2)$ of Lemma \ref{product}, the image of $\ideal  Q$ is topologically scattered. Since $f$ is an embedding, $g$ is surjective (indeed, $I= g(\downarrow f(I))$ for every $I\in \ideal  P$). Hence $\ideal  P$ is topologically scattered as required.

\noindent $(iii) \Rightarrow (vi)$. Obvious: $\overline{\mi_{\Delta}(\ideal  P)}$ is a subset of $\ideal  P$. 

\noindent $(vi) \Rightarrow (v)$.  Since $\overline{\mi_{\Delta}(\ideal  P)}$ is topologically scattered it is order-scattered (see Corollary \ref{top->order}). Hence $\mi_{\Delta}(\ideal  P)$  is order-scattered. 

\noindent $(v) \Rightarrow (iv)$. The proof follows the same lines as the proof of Theorem \ref{duality}. Since $dim_\vee(P)=n< \infty$, by Theorem \ref{cor:dimension} $cov(\mi_{\Delta}(\ideal  P))=n$. Let us cover $\mi_{\Delta}(\ideal  P)$ by $n$ chains $C_k$, $k<n$. Close each $C_k$ by intersection. Let $\overline C_k$ be the resulting chain. Since $\mi_{\Delta}(\ideal  P)$ is order-scattered, $ C_k$ is order-scattered. Moreover,  $\overline C_k$ is order-scattered. Hence, by Lemma \ref{q in product},  $L:= \Pi_{k\in K} \ideal  C_k$ is order-scattered. For $I\in \ideal P$ let  $I_k$ be the least element $J$ of $\overline C_k$ such that $I\subseteq J$ and let $f(I):= (I_k)_{k\in K}$. The map $f$ is an embedding of $\ideal  P$ into $L$.  Hence $\ideal  P$ is order-scattered.

\noindent $(iv) \Rightarrow (i)$. Let $f$ be a join-embedding of $P$ into a  product of $n$ chains,  say $Q=\Pi_{k<n} C_k$. For $k<n$, let $f_k$ be the projection map from $P$ onto $C_k$. W.l.o.g we may suppose each $f_i$ surjective. Let $I\in \ideal  C_i$. Since $f$ is join-preserving, $f^{-1}(J)\in \ideal  P$. Let   $g_k$ be the map from $\ideal  C_k$ to $\ideal  P$ defined by setting $g_k(I):= f_k^{-1}(I)$. This map is order-preserving and one-to-one. Since $\ideal  P$ is order-scattered, $\ideal  C_k$ is order-scattered, hence $C_k$ is order-scattered. 

\end{proof}

\begin{rem} The following question (see \cite{PoSiZa2}) is unanswered. Does a poset  $P$ which embeds in a product of $k$ scattered chains and in a product of $n$ chains embeds into a product of $n$ scattered chains?   Equivalence between $(i)$ and $(ii)$ states that the answer is positive if we consider join-semilattices and join-embeddings. \end{rem}

The proof of  Theorem \ref{fin dim} below is based on  a famous unpublished result of Galvin about partitions of pairs of the rationals which can be stated in terms of the "bracket" relation by:
\begin{equation}\label{eq:bracket}
\eta \rightarrow [\eta]_{2}^2. 
\end{equation} 
This relation means   that if the pairs of rationals are divided into finitely many classes then there is an infinite subset of the rationals which is isomorphic to the rationals and such that all pairs belong to the union of two classes. 

An alternative statement is the following: 

 \begin{thm}\label{Galvin}
Suppose the pairs of rationals be  divided into finitely many classes $A_1,\dots,A_n$. Fix an  ordering on the rationals with order type $\omega$. Then there is a subset $X$ of rationals of order type $\eta$ and indices $i,j$ (with possibly $i=j$) such that all pairs of $X$ on which the natural order on $\mathbb Q$  and the given order  coincide belong to $A_i$, and all pairs of $X$ on which the two orders disagree belong to $A_j$.
\end{thm}

The proof of Galvin's result can be found in \cite{fraisse}; Theorem \ref{Galvin} was used in \cite{PoSiZa}.

\begin{thm}\label{fin dim} Let P be a  join-semilattice with $0$ and $L:=\ideal  P$ be the lattice  of ideals of $P$ ordered by inclusion. If $dim_\vee P=n<\omega$, then $L$ is order-scattered iff $P$ is order-scattered and $\Omega(\eta)$ is not a join-subsemilattice of $P$.
\end{thm}
\begin{proof}

Let $Q\in \{\mathbb Q, \Omega(\eta)\}$. If $Q$ is embeddable into $P$  then $\ideal  Q$ is embeddable into $\ideal  P$. Since $\ideal  Q$ contains a chain of order-type $\eta$, $\ideal  P$ is not order-scattered. This proves that if $L$ is order-scattered then $P$ is order-scattered and $\Omega(\eta)$ is not a join-subsemilattice of $P$.

For the converse, we find convenient to view  $L$ as the lattice $C:=Cl(X, \phi)$ of an algebraic closure system $(X, \phi)$ such that $\emptyset \in C$. The join-semilattice $P$ is isomorphic to $K$, the collection of finitely generated closed subsets of $(X, \phi)$. 
 
Since $dim_\vee P=n<\omega$, there are $n$ maximal chains $C_k$, $k<n$,   of $L$ whose union covers $M:=\mi_{\Delta}(\ideal  P)$, see Proposition \ref{prop:finite-dimension}.  By Lemma  \ref{Ci} for each $k$ there is some total quasi-order $\leq_k$ such that $C_k=I(X, \leq_k)$. Now suppose that $C$ is not order-scattered. By Lemma  \ref{join}  $\Pi_{k\in K}C_k$ is not scattered. By Lemma \ref{q in product} some $C_k$ is not scattered. Without loss of generality we assume that $C_0$ is not scattered.  

%

\begin{claim}\label{claim:I(eta)}
 $X_0:= (X,\leq_0)$ is not scattered.
\end{claim}
\noindent {\bf Proof of Claim \ref{claim:I(eta)}}. 
According to Claim \ref{Ci}, $C_0= I(X_0)$. The fact that $C_0$ is not scattered implies that $X_0$ is not scattered. This is a well known fact. For reader convenience, we give a proof. We use the fact that the chain of rationals contains a copy of every countable chain (Cantor), hence a copy of a chain  of order type $2\cdot \eta$. Thus 
 $C_0$ contains  a sub-chain $D$ of order type $2\cdot \eta$. We may write the elements  of 
$D$ as $d_{rs}$ with  $r \in \mathbb{Q},s \in \{0,1\}$, these elements  being ordered  by  $d_{rs}<_0d_{r's'}$  if $r<r'$, or $r=r'$ and $s=0,s'=1$. For each $r \in \mathbb{Q}$, pick  $x_r \in X$ such that $d_{r0}\leq_0 x_r \leq_0 d_{r1}$. If $r<r'$ are two rationals, then $x_{r}\leq_0 x_{r'}$, but $x_{r'}\not \leq_0 x_r$. Hence, $x_r <_0x_{r'}$. 
Thus, the set of  $x_r$'s forms a chain of type $\eta$ in  $X_0$.  							\hfill $\Box$

Let $A \subseteq X$  such that the order $\leq_0$  induced on $A$ has  order type $\eta$ and let $\leq_n$ be a linear order of type $\omega$ on $A$. We denote by $[A]^2$ the set of pairs of distinct elements of $A$ and identify  each such pair to an  ordered pair $(x,y)$ such that 
%
%
$x<_0y$. To each $u:=(x,y)\in A^2$, we assign a sequence $\epsilon (u):=  \langle \epsilon_1(u),\dots,\epsilon_{n}(u)\rangle$, where $\epsilon_k(u)=1$ if $x<_ky$,  $\epsilon_k(u)=0$ if  $x>_ky$,  and $\epsilon_k(u)=2$ if $x\leq _ky$ and $y\leq_kx$. 

Since $\epsilon_n(u)\not = 2$,  this defines a  partition of distinct pairs of $A$ into at most $3^{n-1}.2$ classes. 

\begin{claim}\label{claim:galvin} There is a subset $A'$ of $A$ such that the order $\leq'_0$ induced on $A'$ by $\leq _0$ has order type $\eta$ and  such that for each  $k$, $1\leq k<  n$, the restriction $\leq'_k$ of  $\leq_k$ to $A'$  is either $\leq'_0$, its dual or $\leq' _{n}$ (the restriction to $A'$ of $\leq_n$), its dual, or $A'^2$, the complete relation on $A'$.
\end{claim}
\noindent {\bf Proof of Claim \ref{claim:galvin}.}
According to the bracket relation in (\ref{eq:bracket})  one can find a subset $A'\subseteq A$, that has type $\eta$ with respect to $\leq_0$ such that the range of the restriction of $\epsilon$ to distinct pairs of $A'$ has at most two elements. Hence, there are  
  $n$-sequences $\alpha:=\langle \alpha_1, \dots, \alpha_n\rangle$ and $\beta:=\langle \beta_1, \dots, \beta_n\rangle$  such that  $\epsilon (u)\in \{ \alpha, \beta\}$ for every $u$ (and the values $\alpha$ and $\beta$ are attained). 
  Necessarily, $\alpha_n\not =\beta_n$. Indeed, otherwise, $\epsilon_n(u)$ would be constant, meaning that $\leq'_n$ to $A$ would coincide with $\leq'_0$ or with its dual; this is impossible since $\leq'_0$ has order type $\eta$.  Hence $0,1\in \{ \alpha_n, \beta_n\}$. With no loss of generality, we may assume that $\alpha_n=1$ and $\beta_n=0$. Consequently, for every $u$ and every $k$, $1\leq k< n$,  we have:
  \begin{equation}\label{eq:epsilon1}
  \epsilon_k(u)=\alpha_k\;  \text{if}\;   \epsilon_n(u)=1
  \end{equation}
  and 
  \begin{equation}\label{epsilon2}
  \epsilon_k(u)=\beta_k \:\text{if}\;  \epsilon_n(u)=0. 
  \end{equation}

  Furthermore, we have: 
  
  \begin{equation}\label{epsilon3}
 \alpha_k=2 \;  \text{if and only if}\;  \beta_k=2. 
  \end{equation}

Indeed, suppose $\alpha_k=2$ but  $\beta_k\not =2$. W.l.o.g. we may suppose $\beta_k=1$. Thus for every $u\in [A']^2$, we have $\epsilon_k(u)=2$ if $\epsilon_n(u)=1$ and $\epsilon_k(u)=1$  if $\epsilon_n(u)=0$. This is impossible. Since the order $\leq_n$ has type $\omega$ we may find $x<_{0} z<_{0}y$ such that $z<_{n}x<_{n}y$. We have  $\epsilon_k(z,y)=2$ and $\epsilon_k(x,y)=2$. By transitivity of $\leq_k$ we have $\epsilon_k(x,z)=2$ which contradicts $\epsilon_k(x,z)=1$. 

Now, we may conclude: 
if $\alpha_k=2$ then $\beta_k=2$, hence, $\epsilon_k(u)=2$ for every $u\in [A']^2$ that is  $\leq'_k$ is  $A'^2$, the complete relation on $A'$.  If $\alpha_k=1$, then either $\beta_k=1$, or $\beta_k=0$. In the first case,   $\leq'_k$ is $\leq'_0$; in the second case, $\leq'_k$ is $\leq'_n$. Similarly, if $\alpha_k=0$, $\leq'_k$ to $A'$ is the dual of $\leq'_0$ if $\beta_k=0$, whereas it is the dual of $\leq'_n$ if $\beta_k=1$. \hfill $\Box$

Let $C'$ be the closure induced on $A'$, that is $C':= (A', \phi')$ where $\phi'(Z):= \phi(Z)\cap A'$ for every $Z\subseteq A'$. 
 
 The quasi-orders of the form $A'^2$ play no role in $C'$. Hence  from Claim \ref{claim:galvin}: 
 
\begin{claim}\label{claim:convex} 
There exists a subset $\mathcal L\subseteq  \{ (\leq'_0)^*, \leq'_n,  (\leq'_n)^*\}$ such that  $C'$ is the join of the $I(A', \leq')$ where $\leq'$ belongs to $\mathcal L\cup \{\leq'_0\}$.
\end{claim}

Since the join-semilattice $K(C')$ of compact elements of $C'$ embeds into $K(C)$ it suffices to prove that it embeds $\eta$ or $\Omega(\eta)$. 

There are $8$ cases to consider, $4$ yield an embedding of $\eta$, the $4$ remaining yield an embedding of $\Omega(\eta)$:

\begin{claim}\label{claim:cases} 
\begin{enumerate}

 \item  $\eta$ embeds into $K(C')$ if $\mathcal L$ does not contain $\{\leq'_{n}\}$, that is $ \mathcal L$ is either empty or equal to $\{(\leq'_0)^*\}$ or to $\{(\leq'_n)^*\}$ or  to $\{(\leq'_0)^*, (\leq'_n)^*\}$. 
\item $\Omega (\eta)$ embeds into $K(C')$  if $\mathcal L$  contains $\{\leq'_{n}\}$, that is $\mathcal L$ is equal to $\{\leq'_n\}$,  $\{\leq'_n,(\leq'_n)^*\}$ or $\{(\leq'_0)^*, \leq'_n\}$   or $\{\leq'_0, \leq'_n,(\leq'_n)*\}$. 
\end{enumerate}
\end{claim}
{\bf Proof of Claim \ref{claim:cases}.}
 We define a subset of $D$ which is isomorphic to $\eta$ or embeds  $\Omega(\eta)$. 
If $\mathcal L$ is empty, we  set $D:=\{\phi'(\{x\}): x\in A'\}$. In all other cases, $D$ is of the form $\{\phi'(\{x_0, x\}): x_0<'_0 x: x\in A'\}$ for some $x_0\in A'$. If  $\mathcal L=\{(\leq'_0)^*\}$, $x_0$ is arbitrary; in all other cases $x_0$ is the   least  element of $\leq'_n$. 

\noindent Item (1). If $\mathcal L=\emptyset$ then, trivially, $D$ is isomorphic to $(A', \leq'_0)$, hence, has order type $\eta$. Next,  observe that  $\phi'(\{x_0, x\})=\{y\in A': x_0\leq y\leq'_0 x\}$ if $\mathcal L$ is $\{(\leq'_0)^*\}$ or $\{(\leq'_0)^*, (\leq'_n)^*\}$, and $\phi'(\{x_0, x\})=\{y \in A': y\leq'_0 x\}$, if  $\mathcal {L}$  is $\{(\leq'_n)^*\}$. Hence,  $D$ forms a chain of type $\eta$. 

\noindent Item (2). If $\mathcal L=\{\leq'_n\}$ then $K(C')$ embeds $\Omega (\eta)$ by Theorem  \ref{mainduplex}. In fact, in this case,  as in case $\mathcal L= \{\leq'_n,(\leq'_n)^*\}$, we have $\phi'(\{x_0, x\})=\{y\in A': y\leq'_0 x \; \text{and}\;  y\leq'_n x   \}$. 
If $\mathcal L$ is $\mathcal L=\{(\leq'_0)^*, \leq'_n\}$  or if $\mathcal L= \{\leq'_0, \leq'_n,(\leq'_n)*\}$, then $\phi'(\{x_0, x\})=\{y\in A': x_0\leq'_0 y \leq'_0 x \; \text{and}\;  x_0\leq'_n y \leq'_n x   \}$. Hence, $D$ embeds $\Omega(\eta)$.

\hfill $\Box$

With this last claim, the proof of Theorem \ref{fin dim} is complete. 
\end{proof}

\begin{pr} 
As it stands, Theorem \ref{fin dim} does not allow to prove Theorem \ref{CP-like}. Extend the conclusion of Theorem \ref{fin dim} to join-semilattices of finite order-dimension. 
\end{pr}


\begin{thebibliography}{10}


\bibitem{Ada}
K.~V.~Adaricheva,
\emph{Semidistributive and co-algebraic lattices of subsemilattices},
Algebra and Logic \textbf{27} (1988), 385--395; translated from
Algebra i Logika \textbf{27}, no.~6 (1988), 625--640.

\bibitem{A04}
K.~Adaricheva, \emph{Join-semidistributive lattices of relatively convex sets},
Contributions to General Algebra 14, Proceedings of the Olomouc Conference 2002 (AAA 64)
and the Potsdam conference 2003 (AAA 65), Verlag Johannes Heyn, Klagenfurt, 2004, 1--14.
Also see on http://arxiv.org/abs/math/0403104   [Math.RA]

\bibitem{A08}
\bysame, \emph{Representing finite convex geometries by relatively convex sets}, European J. Combin. 37 (2014), 68--78.
\bibitem{A09} 
\bysame, \emph{On prevariety of perfect lattices}, Alg. Universalis \textbf{65} (2011), 21--39.

\bibitem{A14} \bysame, \emph{Algebraic convex geometries revisited}, arxiv 2014.

\bibitem{AdaGorTum}
K.V.~Adaricheva, V.A.~Gorbunov, and V.I.~Tumanov,
\emph{Join-semidistributive lattices and convex geometries},
Adv. Math., \textbf{173} (2003), 1--49.

\bibitem{adaricheva-nation} K.V.~Adaricheva, J.B.~Nation, A class of infinite convex geometries, 10 p., arXiv:1501.04174v1 [math.CO] 17 jan 2015. 

\bibitem{A07} D.~Armstrong, \emph{The sorting order on a Coxeter group},  J. Combin. Theory Ser. A 116 (2009), no. 8,  285--1305.

 \bibitem{beagley} J.E.~Beagley,  On the order dimension of convex geometries. Order 30 (2013), no. 3, 837--845.
\bibitem{Be}  G.M.~Bergman, \emph{On lattices of convex sets in $R^n$}, Algebra Universalis \textbf{53} (2005), 357--395.

\bibitem{BonPou} R.~Bonnet, and M.~Pouzet, \emph{Linear extensions of ordered sets}, Bannf, Alta., 1981, 125--170, NATO Adv.Study Inst. Ser. C: Math.Phys. Sci. 83, Reidel, Dordrecht-Boston, Mass. 1982.
\bibitem{birkhoff-bennett}  G.~Birkhoff  and M.K.~Bennett, The convexity lattice of a poset, Order 2 (1985) 223-242.

\bibitem{brignall}
R.~ Brignall,
\newblock {A survey of simple permutations. Permutation patterns.}
\newblock 41-65, London Math. Soc. Lecture Note Ser., 376, Cambridge Univ. Press, Cambridge, 2010.


\bibitem{ChaPou} I.~Chakir, and M.~Pouzet, \emph{Infinite independent sets in distributive lattices}. Algebra Universalis 53 (2005), no. 2-3, 211--225. 

\bibitem{ChaPou0}
I.~Chakir, and M.~Pouzet, \emph{The length of chains in modular algebraic lattices}, Order, 24(2007), 227--247.

\bibitem{ChaPou1}
I.~Chakir, and M.~Pouzet, \emph{The length of chains in algebraic lattices}, Proceedings ISO-08, Algiers, Algeria, Nov 2-6, 2008, H.~Ait Haddadene, I.~Bouchemakh, M.~Boudar, S.~Bouroubi (Eds) LAID3, p. 379--390.  http://arxiv.org/abs/0812.2193v1 [math.CO]

\bibitem{ChaPou2}
I.~Chakir, and M.~Pouzet, \emph{A characterization of well-founded algebraic lattices}, preprint, arXiv:0812.2300 [math.CO].

\bibitem{cameron} P. J~Cameron, \newblock{Homogeneous permutations. Permutation patterns} (Otago, 2003). Electron. J. Combin. {\bf 9} (2002/03), no. 2, Research paper 2, 9 pp.
\bibitem{delhomme-zaguia} C.~Delhomm\'e, I.~Zaguia, Countable linear orders with disjoint infinite intervals are mutually orthogonal, preprint, Universit\'e de la R\'eunion,  december 2014. 
\bibitem{D}
B.~Dietrich, \emph{A circuit characterization of antimatroids}, J. Comb. Theory B \textbf{43}(1987), 314--321.

\bibitem{EdJa}
P.H.~Edelman and R.~Jamison, \emph{The theory of convex geometries}, Geom.Dedicata \textbf{19}(1985), 247-274.

\bibitem{ehren}
A.~ Ehrenfeucht, T.~Harju, G. Rozenberg,
 \newblock{\em The theory of 2-structures. A framework for decomposition and transformation of graphs.}
\newblock{\em World Scientific Publishing Co., Inc., River Edge, NJ, 1999.}

\bibitem{fraisse}
M.~R. Fra\"{\i}ss\'{e},
\newblock {\em Theory of relations}.
\newblock Second edition, North-Holland Publishing Co., Amsterdam, 2000.


\bibitem{Comp}
G.~Gierz, K.H.~Hofmann, K.~Keimel, J.D.~Lawson, M.~Mislove, and D.S.~Scott,
``Continuous lattices and domains'',
Encyclopedia of Mathematics and its Applications 93, Cambridge University Press, 2003.

\bibitem{gratzer}G.~Gr\"atzer, Lattice theory: foundation. Birkh\"auser/Springer Basel AG, Basel, 2011. xxx+613 pp. 

\bibitem{HR93} M.~Hawrylicz, and V.~Reiner, \emph{The lattice of closure relations on a poset}, Algebra Universalis \textbf{30}(1993), 301--310.
 \bibitem{kierstead-milner} H. A.~Kierstead, E. C.~Milner, The dimension of the finite subsets of $\kappa$. Order 13 (1996), no. 3, 227--231.
 
 \bibitem{klazar}
M.~Klazar, \newblock Overview of general results in combinatorial enumeration, in {\em Permutation patterns},
London Math. Soc. Lecture Note Ser.,
Vol. 376, (2010), 3--40, Cambridge Univ. Press, Cambridge.
\bibitem{laflamme}
C.~Laflamme, M.~Pouzet, N.~Sauer, I.~Zaguia, \newblock Pairs of orthogonal countable ordinals, J. Discrete Math  335 (2014), 35--44.


\bibitem  {lmp3}J.D.~Lawson, M.~Mislove, H. A.~Priestley, Ordered sets with no infinite antichains,
Discrete Mathematics, 63 (1987), 225--230.

\bibitem{Marcus}
A.~ Marcus, G.~Tard\"os,   \newblock {\em Excluded permutation matrices and the Stanley-Wilf conjecture}, {\em J. Combin. Theory}, {Ser. A} {\bf 107} (2004), 153--160.

\bibitem{Mis} M.~Mislove, \emph{When are topologically scattered and order scattered are the same?},
Annals of Disc. Math.  23(1984), 61--80.

\bibitem{Mor} W.~Morris, and V.~Soltan, \emph{The Erd\H{o}s-Szekeres problem on points in convex position-A survey}
Bulletin of the AMS 37(2000), 437--458.
\bibitem{Mor2} W.~Morris,  \emph{Coloring copoints of a planar point set}. Discrete Appl. Math. 154 (2006), no. 12, 1742--1752. 
\bibitem{nozaki} A.~Nozaki, M.~Miyakawa, G.~Pogosyan, I.G.~Rosenberg, The number of orthogonal permutations, European J. Combin. 16 (1995) 71--85.


\bibitem{Nat} J.B.~Nation, \emph{Notes on Lattice Theory}, available at http://www.math.hawaii.edu/~jb/lat1-6.pdf
\bibitem{Pel} A.~Pelczynski and Z.~Semadeni, {\it Spaces of continuous 
functions,III. \/} Studia Math. {\bf 18} (1959), 
211--222.
\bibitem{pouzet-zaguia} M.~Pouzet,  N.~Zaguia, \emph{Ordered sets with no chains of ideals of a given type}. Order 1 (1984), no. 2, 159–172.

\bibitem{PoSiZa} M.~Pouzet, H.~Si Kaddour and N.~Zaguia, \emph{Which posets have a scattered MacNeille completion?}, Alg. Universalis  53 (2005), 287--299.

\bibitem{PoSiZa2} M.~Pouzet, H.~Si Kaddour, N.~Zaguia,  \emph{Finite dimensional scattered posets}. European J. Combin. 37 (2014), 79–99.

\bibitem{SW11} L.~Santocanale and F.~Wehrung, \emph{Varieties of lattices with geometric description},   http://arxiv.org/abs/1102.2195  [Math.RA]

\bibitem{sauer-zaguia} N.~Sauer,  I.~Zaguia, \emph{The order on the rationals has an orthogonal order with the same order type}. Order 28 (2011), no. 3, 377--385.
\bibitem{SW4} M.~Semenova, and F.~Wehrung, \emph{Sublattices of lattices of convex sets of vector spaces},  (Russian) Algebra i Logika \textbf{43}(2004), 261-290; translation Algebra and Logic \textbf{43}(2004), 145--161.

\bibitem{SZ07} M.~Semenova, and A.~Zamojska-Dzienio, \emph{On lattices embeddable into lattices of order-convex sets. Case of trees}, Internat. J.Algebra Comput. \textbf{17} (2007), 1667-1712. 

\bibitem{S05} M.V.~Semenova  \emph{On lattices embeddable into lattices of suborders}, (Russian) Algebra i Logika \textbf{44} (2005),483--511; translation in Algebra and Logic \textbf{44} (2005), 270--285.

\bibitem{trotter} W.T.~Trotter,\newblock{\em  Combinatorics and Partially Ordered Sets: Dimension Theory,} The Johns
Hopkins University Press, Baltimore, MD, 1992.

\bibitem{W01} N.~Wahl, \emph{Antimatroids of finite character}, J. Geometry \textbf{70} (2001), 168--175.

\bibitem{W05} F.~Wehrung, \emph{Sublattices of complete lattices with continuity conditions},
Algebra Universalis \textbf{53} (2005), 149--173.
\end{thebibliography}
\end{document}